\documentclass[10pt]{article}
 
\usepackage{amsmath}
\usepackage{amssymb}
\usepackage{amsthm}
\usepackage{latexsym}
\usepackage{rotate,epsfig}
\usepackage{lscape}
\usepackage{color}

\usepackage{mathptmx}    
\usepackage{latexsym}
\newtheorem{definition}{Definition}[section]
\newtheorem{proposition}{Proposition}[section]
\newtheorem{remark}{Remark}[section]
\newtheorem{theorem}{Theorem}[section]
\newtheorem{lemma}{Lemma}[section]
\newtheorem{example}{Example}[section]

\begin{document}

\title{\bf Weighted mean inactivity time function with applications\thanks{Research partially supported by 
Gonbad Kavous University, Gruppo Nazionale per il Calcolo Scientifico dell'Istituto Nazionale 
di Alta Matematica ``F.\ Severi'', and MIUR (PRIN 2017, Project 
``Stochastic Models for Complex Systems'', no.\ 2017JFFHSH).}
}


\author{
Antonio Di Crescenzo\footnote{
Dipartimento di Matematica, Universit\`a di Salerno, 
Via Giovanni Paolo II n.132, I-84084 Fisciano (SA), Italy, 
Email: adicrescenzo@unisa.it -- ORCID: 0000-0003-4751-7341}
\and
Abdolsaeed Toomaj\footnote{
Faculty of Basic Sciences and Engineering, Department of Mathematics and Statistics, 
Gonbad Kavous University, Gonbad Kavous, Iran, 
Email:  ab.toomaj@gonbad.ac.ir and ab.toomaj@gmail.com -- ORCID: 0000-0001-8813-473X}
}

\maketitle

\begin{abstract}
The concept of mean inactivity time plays a crucial role in reliability, risk theory and life testing. In this regard, we introduce a weighted mean inactivity time function by considering a non-negative weight function. Based on this function, we provide expressions for the variance of transformed random variable and the weighted generalized cumulative entropy. The latter concept is an important measure of uncertainty which is shift-dependent and is of interest in certain applied contexts, such as reliability or mathematical neurobiology. Moreover, based on the comparison of mean inactivity times of a certain function of two lifetime random variables, we introduce and study a new stochastic order in terms of the weighted mean inactivity time function. Several characterizations and preservation properties of the new order under shock models, random maxima and renewal theory are discussed.


\end{abstract}

\section{Introduction and preliminaries}

Over the last decades, various concepts of stochastic orders have been defined and studied in the literature for the sake of their useful applications in reliability and economics and as mathematical tools for proving important results in applied probability. A comprehensive discussions on many stochastic comparisons between random variables are reported and investigated in details in the monograph given by Shaked and Shanthikumar \cite{Shaked}. The mean inactivity time (MIT) function, also known as the mean past lifetime and the mean waiting time, is a well-known reliability measure which has many applications in various disciplines such as reliability theory, survival analysis, risk theory and actuarial studies, among others.
\par
Let $X$ be a non-negative absolutely continuous random variable denoting the lifetime of a system or a component or a living organism with cumulative distribution function (CDF) $F(x)=\mathbb{P}(X\leq x)$ and probability density function (PDF) $f(x).$ Under the condition that the system has been found failed before time $t,$ the inactivity time is defined by $X_{[t]}=[t-X \, | \,X\leq t].$ In fact, $X_{[t]}$ denotes a random variable whose distribution is the same as the conditional distribution of $t-X$ given that $X\leq t.$ It is worth emphasizing that in many realistic situations, the random lifetime can refer to the past. For instance, consider a system whose state is observed only at certain preassigned inspection times. If at time $t,$ the system is inspected for the first time and it is found to be ``down", then the failure relies on the past (see e.g.\ Kayid and Ahmad \cite{KayidAhmad2004} and Di Crescenzo and Longobardi \cite{Di-Longobardi-2009}). Now, we recall the MIT function of $X$ which is defined by
\begin{equation}\label{mit}
\widetilde{\mu}(t)=\mathbb{E}[t-X|X\leq t]
=\frac{1}{{F}(t)}\int_{0}^{t}{F}(x)\,{\rm d}x,
\qquad t\in D,
\end{equation}
where $D:=\{t>0:{F}(t)>0\}$ and where  $\mathbb{E}[\,\cdot\,]$ means expectation.
An interpretation of the MIT function is as follows:
Assume that at time $t$ we perform an inspection to a device  and then we realize that
it has already failed. The MIT is thus useful  to infer on the actual time at which the failure of the device
occured. For further interpretations we refer the readers to e.g.\ Kayid and Izadkhah \cite{KayidIzadkhah2014}.
Assuming that $\widetilde{\mu}(t)$ is a differentiable function, from (\ref{mit}) we get
\begin{equation}\label{dmutilde}
 \widetilde{\mu}'(t)=1-\tau(t)\widetilde{\mu}(t),\qquad t\in D,
\end{equation}
where
\begin{equation}\label{rhr}
\tau(x)=\frac{f(x)}{F(x)},\qquad x\in D,
\end{equation}
denotes the reversed hazard rate function of $X.$ It is known that the reversed hazard rate and the MIT functions under certain assumptions define uniquely $F(t)$ as follows:
\begin{equation}\label{uniq}
F(t)=\exp\left\{-\int_{t}^{\infty}\tau(x)\,{\rm d}x\right\}
=\exp\left\{-\int_{t}^{\infty}\frac{1-\widetilde{\mu}'(x)}{\widetilde{\mu}(x)}\,{\rm d}x\right\},
\qquad t\in D.
\end{equation}
As pointed out by Finkelstein \cite{Finkelstein}, relation (\ref{uniq}) characterizes distribution functions if the following statements hold: 
\par
(i) $\widetilde{\mu}(0)=0$ and $\widetilde{\mu}(x)>0$ for all $x>0,$ 
\par
(ii) $\widetilde{\mu}'(x)<1,$ 
\par
(iii) $\int_{0}^{\infty}({1-\widetilde{\mu}'(x)})/{\widetilde{\mu}(x)}\,{\rm d}x=\infty$,
\par
(iv) $\int_{t}^{\infty}({1-\widetilde{\mu}'(x)})/{\widetilde{\mu}(x)}\,{\rm d}x<\infty,$ for all $t>0.$
\\
It follows from Eq.\ (\ref{uniq}) and characterization conditions for $\widetilde{\mu}(x)$ that there is no lifetime distribution with decreasing MIT function. Indeed, $\widetilde{\mu}'(x)<0,$ in this case and condition (iv) does not hold (see Finkelstein \cite{Finkelstein} for further details).
Kayid and Ahmad \cite{KayidAhmad2004} (see also Ahmad et al. \cite{Ahmad}) studied stochastic comparisons based on the MIT function under the reliability operations of convolution and mixture. Badia and Berrade \cite{Badia} gave an insight into properties of the MIT in mixtures of distributions. Some further properties of MIT function are widely studied and investigated in Finkelstein \cite{Finkelstein}, Goliforushani and Asadi \cite{Goli}, Kundu \emph{et al.} \cite{Kundu} and the references therein. Moreover, Izadkhah and Kayid \cite{IzadkhahKayid} used the harmonic mean average of the MIT function to propose a new stochastic order. 
Recently, Toomaj and Di Crescenzo \cite{ToomajAD2020}
showed that the variance of a random variable can be represented in terms of the square of the MIT function. 
Therefore, it is not  surprising that the MIT function has been the object of several investigations.
The aim of the present paper is to define a new version of MIT function, namely the weighted MIT function 
and to show some applications of such a measure. 
In analogy with (\ref{mit}), the weighted MIT function is defined through   the expectation 
$\mathbb{E}[\psi(t)-\psi(X)|X\leq t]$, where $\psi$ is a suitable cumulative weight function. 
By means of suitable choices of $\psi$ we show that the weighted MIT function can be related to 
various notions of reliability theory, as well as to several information measures of interest, 
such as the dynamic cumulative entropy, the past entropy, the varentropy and the weighted cumulative entropy. 
In other terms, the introduction of the weight function $\psi$ allows to construct a flexible tool which  
unifies various notions emerging in different applied fields. Moreover, the generalized MIT function 
can be used to extend the MIT stochastic order to the weighted version. Among the  main 
theoretical contributions of this paper, we refer to the introduction of the left spread function, 
which is analog to the right spread function (also known as excess wealth transform). 
The given function, which is of interest in risk management, is also extended to the weighted 
version. The latter one is shown to be intimately related to the variance of the weighted random 
variable and to the weighted generalized cumulative entropy. 
\par
Therefore, the rest of this paper is organized as follows:
In Section 2, some general properties of the weighted mean inactivity time function are discussed. 
We provide suitable connections with the ROC curve. We also analyze conditions 
expressed in terms of the reversed hazard rate function 
such that the weighted mean inactivity time function is constant, and also that it is increasing.
Section 3 is devoted to introduce the weighted mean inactivity time order. 
We also analyze its properties and connections to other well-known stochastic orders. 
In particular, we find additional conditions that allow to relate this order with the 
reversed hazard rate order. 
In Section 4, we use the weighted MIT function to obtain expressions and various results for the variance 
of transformed random variables as well as for the weighted generalized cumulative entropy. 
Furthermore, attention is  given to the determination of bounds and to the 
representation of measures in terms of expectations. 
Section 5 is finalized to investigate some connections of the previous results with the location-independent 
riskier order. In Section 6 we focus on applications of the previous results 
to reliability theory, with special attention to ordering results for 
a shock model governed by a non-homogeneous Poisson process, and 
for the maximum of independent and identically distributed (i.i.d.)\ random variables. 
Finally, we provide also applications to renewal theory
based on the weighted mean inactivity time order, 
with emphasis on the excess lifetime of a customary renewal process.
\par
For simplicity, in the rest of the paper, we write $g^n(x)$ instead of $[g(x)]^n$ for any given function $g$.
Moreover, $g'(x)$ denotes the derivative of $g(x)$. Note that the terms increasing and decreasing are
used in nonstrict sense. Throughout this paper, it is assumed that the expectations exist when they appears.
Furthermore, we denote by $\sigma^2(X)$ or $Var(X)$ the variance of $X$.
In addition, given two subsets of the real line $\cal U$ and $\cal V$, we say that  a non-negative function
$K(u,v)$ defined on $\cal U\times \cal V$ is {\em totally positive of order 2}, denoted as $TP_2$, if $K(u_1, v_1)K(u_2, v_2) \geq K(u_1, v_2)K(u_2, v_1)$ for all
$u_1\leq  u_2$ in $\cal U$ and $v_1 \leq v_2$ in $\cal V$ (see Karlin \cite{Karlin}). Finally,
``$\log$" denotes the natural logarithm, with the convention $0\log 0=0$.
\section{Weighted Mean Inactivity Time Function}

The aim of this section is to investigate on the weighted mean inactivity time function by applying \emph{the cumulative weight function}, say. For this aim, we consider a non-negative and differentiable function $\phi(x)$ in $[0,\infty).$ The cumulative weight function is defined as
\begin{equation}\label{weight function}
 \psi(x):=\int_{0}^{x}\phi(u)\,{\rm d}u,
 \qquad x\geq 0.
\end{equation}
This function plays a pivotal role in achieving our results. Specifically, given the random lifetime $X,$ we analyze various properties of the transformed random variable $\psi(X)$, where the latter may be viewed as an increasing time-change of $X.$ Let $\overline{F}(t)=1-F(t)$ be the survival function of $X$, and let
\begin{equation}
 \lambda(x)=\frac{f(x)}{\overline{F}(x)},
 \qquad  \forall\; x\geq 0: \overline{F}(x)>0
 \label{eq:failrate}
\end{equation}
denote the hazard rate function of $X$. For example, if we consider $\phi(x)=\lambda(x),$ then (\ref{weight function}) gives the cumulative hazard function of $X$. Due to (\ref{weight function}), it is clear that $\psi(x)$ is an increasing function of $x>0$ such that $\psi(0)=0,$ since $\psi'(x)=\phi(x)\geq 0$. Additionally, if the weight function $\phi(x)$ is increasing (decreasing) in $x>0,$ then $\psi(x)$ is convex (concave). This function was successfully applied by Toomaj and Di Crescenzo \cite{ToomajAD2020Mathematics}
to provide expressions for the variance of cumulative weighted random variable $\psi(X)$
by defining the {\em weighted mean residual life\/} (WMRL) function as
\begin{eqnarray}\label{vark}
 m_{\psi}(t)=\mathbb{E}[\psi(X)-\psi(t)|X>t]
 =\frac{1}{\overline{F}(t)}\int_{t}^{\infty}\phi(x)\overline{F}(x)\,{\rm d}x,
\end{eqnarray}
for all $t\geq 0$ such that $\overline{F}(t)>0$.
In analogy with (\ref{vark}), we introduce the {\em weighted mean inactivity time\/} (WMIT) function as
\begin{eqnarray}\label{wmit}
\widetilde{\mu}_{\psi}(t)=\widetilde{\mu}_{\psi(X)}(t)=\mathbb{E}[\psi(t)-\psi(X)|X\leq t]
 =\frac{1}{{F}(t)}\int_{0}^{t}\phi(x){F}(x)\,{\rm d}x,
\end{eqnarray}
for $t\in D$. In particular, when $\psi(t)=t,$ and hence $\phi(t)=1,$ then Eq.\ (\ref{wmit}) coincides with the MIT function (\ref{mit}),
and Eq.\ (\ref{vark}) gives the mean of the residual lifetime
\begin{equation}
  X_t:=[X-t \, | \,X>t], \qquad  t\in D.
 \label{eq:residuallifetime}
\end{equation}
In what follows, we implicitly assume that
\begin{equation}\label{EPSIX}
\mathbb{E}[\psi(X)]=\int_{0}^{\infty}\psi(x)f(x)\,{\rm d}x<\infty,
\end{equation}
to ensure the finiteness of $\widetilde{\mu}_{\psi}(t)$.
\begin{remark}\rm
Let $Y$ be a random variable with PDF $g(t)$ and CDF $G(t)$.
Let us consider the cumulative weight function $\psi(x)=G(x)$ and hence $\phi(x)=g(x)$.
Clearly, $\psi(X)=G(X)$ takes values in $[0,1]$, with distribution function
$$
 F_{\psi(X)}(u)=\mathbb P[G(X)\leq u]
 =F(G^{-1}(u)), \qquad 0\leq u \leq 1.
$$
This function is related to the Receiver Operating Characteristic (ROC) which was first developed during the Second World War by electrical engineers to analyze radar signals and to study the relation signal/noise, in particular in order to detect correctly enemy objects in battlefields. Recently, the aforementioned function is widely studied by  Cal\`{\i} and Longobardi \cite{CalMaria}. By interchanging the role of $F$ and $G$ in Section 3
of \cite{CalMaria}, the ROC curve has the following representation
$$
{\rm ROC}(u)= \overline F(G^{-1}(1-u)),\qquad 0\leq u \leq 1,
$$
so that ${\rm ROC}(0)=0$ and ${\rm ROC}(1)=1.$ Hence, the relevant
index given by the area under the ROC curve, i.e.\ AUC, is given by (see Section 5 of \cite{CalMaria})
$$
  {\rm AUC}=\int_0^1 {\rm ROC}(u)\,{\rm d}u
=\int_0^{\infty} g(x)\,\overline F(x)\,{\rm d}x
 = \mathbb{E}\left[\overline F(Y)\right],
$$
where we have set  $x=G^{-1}(1-u)$. Clearly, since $\psi(x)=G(x),$ and thus $\phi(x)=g(x),$ from (\ref{vark}) one has ${\rm AUC}=m_{\psi}(0)=m_{G}(0)$.
On the other hand, from (\ref{wmit}) one has also
$$
  {\rm AUC}=1-   \int_0^{\infty} g(x)\, F(x)\,{\rm d}x=1-\widetilde{\mu}_{\psi}(\infty)
  =1-\widetilde{\mu}_{G}(\infty).
$$
Moreover, in this case $\psi(\infty)=1,$ so that
applying Eqs.\ (\ref{wmit}) and (\ref{EPSIX}) one can  obtain another useful representation as follows:
$$
 {\rm AUC}=\mathbb{E}[G(X)].
$$
\end{remark}
\par
Henceforward, we investigate some properties of the WMIT function given in (\ref{wmit}).
To begin with, from Eqs.\ (\ref{weight function}) and (\ref{wmit}) the following lemma is easily obtained.
\begin{lemma}\label{lemmmk}
If $X$ is an absolutely continuous non-negative random variable, then for all $t\in D$
$$
 \widetilde{\mu}_\psi'(t)=\phi(t)-\tau(t)\,\widetilde{\mu}_\psi(t).
$$
\end{lemma}
This result allows to give the condition such that the WMIT function is constant.
\begin{proposition}\rm
Given a constant $c>0$, one has that
$$
 \widetilde{\mu}_\psi(t)=c \qquad 	\hbox{for all }t>0,
$$
if and only if
$$
 \phi(t) =c\, \tau(t)\qquad 	\hbox{for all }t>0.
$$
\end{proposition}
\par
It is evident from (\ref{wmit}) that for an absolutely continuous
non-negative random variable $X,$
the weighted mean inactivity time function for all $t\in D$ can be rewritten as
\begin{eqnarray}\label{wmit2}
\widetilde{\mu}_{\psi}(t)=\psi(t)-\mu_{\psi}(t),
\end{eqnarray}
where
$$
\mu_{\psi}(t)=\mathbb{E}[\psi(X)|X\leq t]=\frac{1}{F(t)}\int_{0}^{t}\psi(x)f(x)\,{\rm d}x,
\qquad t\in D,
$$
is termed as the weighted mean failure time of a system conditioned by a failure before time $t,$ also named `weighted mean past lifetime'. Clearly, the derivative of this function is given by
\begin{equation}\label{derivemu}
\mu'_{\psi}(t)=\tau(t)[\psi(t)-\mu_{\psi}(t)],
\qquad t\in D.
\end{equation}
By virtue of (\ref{uniq}) and using Lemma \ref{lemmmk}, the WMIT function under certain assumptions defines uniquely $F(t)$ as follows:
\begin{equation}\label{uniqweighted}
 F(t)=\exp\left\{-\int_{t}^{\infty}\tau(x)\,{\rm d}x\right\}
 =\exp\left\{-\int_{t}^{\infty}\frac{\phi(x)-\widetilde{\mu}_\psi'(x)}{\widetilde{\mu}_\psi(x)}\,{\rm d}x\right\},
 \qquad t\in D.
\end{equation}
Equation (\ref{uniqweighted}) characterizes the distribution function under the following statements:
\\
(i) $\widetilde{\mu}_\psi(0)=0$ and $\widetilde{\mu}_\psi(x)>0$ for all $x>0;$
\\
(ii) $\widetilde{\mu}_\psi'(x)<\phi(x);$
\\
(iii) $\int_{0}^{\infty}({\phi(x)-\widetilde{\mu}_\psi'(x)})/{\widetilde{\mu}_\psi(x)}\,{\rm d}x=\infty$, and
\\
(iv) $\int_{t}^{\infty}({\phi(x)-\widetilde{\mu}_\psi'(x)})/{\widetilde{\mu}_\psi(x)}\,{\rm d}x<\infty$ for all $t>0$.
\begin{remark}\rm
We remark that from (\ref{uniqweighted}) and the characterization conditions for $\widetilde{\mu}_\psi(x)$,
it follows that there is no lifetime distribution with decreasing WMIT function.
Indeed, recalling (\ref{uniqweighted}), if $\widetilde{\mu}_\psi'(x)<0$ and
\\
(a) if $C(t):=\int_{t}^{\infty}\phi(x){\rm d}x=\infty$ for some $t>0$, then  one has
$$
 \nonumber F(t) < \exp\left\{-\int_{t}^{\infty}\frac{\phi(x)}{{\widetilde{\mu}_\psi(x)}}\,{\rm d}x\right\}
 <\exp\left\{-\frac{C(t)}{\widetilde{\mu}_\psi(t)}\right\}=0,
$$
(b) if $C(t)=\int_{t}^{\infty}\phi(x){\rm d}x<\infty$ for all $t>0,$ then from condition (i) we have
\[
F(0)<\exp\left\{-\frac{C(0)}{\widetilde{\mu}_\psi(0)}\right\}=0.
\]
Hence, in both cases the condition leads to a contradiction, so that there exists no distribution with
decreasing WMIT function.
\end{remark}
It is worth to point out that in some situations the conditions (i)-(iv) may be not satisfied, as shown
in the following example.
\begin{example}
Let $X$ be an absolutely continuous non-negative random variable with PDF $f(x)$ and survival function $\overline{F}(x).$ With reference to (\ref{weight function}), we consider the weight function
$\phi(x)=\overline{F}(x)/F(x)$, also known as odds of survival (see, for instance,
Gupta and Peng \cite{GuptaPeng}), so that from (\ref{wmit}) we get
\begin{equation}\label{mttf}
\widetilde{\mu}_{\psi}(t)=\frac{1}{F(t)}\int_{0}^{t}\overline{F}(x)\,{\rm d}x,
\qquad t\in D.
\end{equation}
Clearly, if $\inf (D)=0$ and $0<f(0)<\infty$, then $\widetilde{\mu}_{\psi}(0)={1}/{f(0)}>0$.
In this case, condition (i) above does not hold and hence the distribution function can not be characterized. The function in the right-hand-side of (\ref{mttf}) is known as the mean time to failure of an item that is
subject to an \emph{age replacement} policy in which a unit is replaced $t$ hours after its installation
or at a failure, whichever occurs first (see Section 3.3 of Barlow and Proschan \cite{Barlow-Proschan} for details). From the latter reference, if $X$ has decreasing (increasing) failure rate, i.e.\ $X$ is DFR (IFR),
then the function  $\widetilde{\mu}_{\psi}(t)$ given in (\ref{mttf}) is increasing (decreasing) in $t$.
This conclusion can also be obtained from Part (i) of Theorem \ref{thm:philambda} by noting that, due to (\ref{eq:failrate}),
\[
\frac{\phi(t)}{\tau(t)}=\frac{1}{\lambda(t)},\qquad t>0,
\]
which is increasing (decreasing) when $X$ is DFR (IFR).
\end{example}
\par
The following result deals with the WMIT and MIT functions.
\begin{lemma}\label{lemmnbwue}
Let $X$ be an absolutely continuous non-negative random variable with weighted mean inactivity time function $\widetilde{\mu}_\psi(t)$ defined as in (\ref{wmit}). If $\psi(x)$ is convex (concave) on $[0,\infty),$ then
\begin{equation}\label{NBWUE}
\widetilde{\mu}_\psi(t)\geq(\leq) \; \psi(\widetilde{\mu}(t)) \qquad \hbox{for all $t\in D$.}
\end{equation}
\end{lemma}
\begin{proof}
By assumption $\psi(x)$ is increasing convex (concave) on $[0,\infty)$, with $\psi(0)=0.$
Thus $\psi(x)$ is superadditive (subadditive), i.e.\ $\psi(z+y)\geq(\leq)\; \psi(z)+\psi(y)$,
for $z,y\geq0$. By substituting $z=x$ and $y=t-x$, with $0\leq x\leq t,$
we obtain $\psi(t)-\psi(x)\geq(\leq)\;\psi(t-x)$ for all $t\geq x\geq0.$
Hence, recalling (\ref{wmit}) and (\ref{mit}) we find that
\begin{eqnarray*}
\nonumber \widetilde{\mu}_\psi(t)&=&\mathbb{E}[\psi(t)-\psi(X)|X\leq t],\\
\nonumber  &\geq(\leq)&\mathbb{E}[\psi(t-X)|X\leq t],\\
  &\geq(\leq)&\psi(\mathbb{E}[t-X|X\leq t])=\psi(\widetilde{\mu}(t)),
  \qquad t\in D,\label{eq:mitphi}
\end{eqnarray*}
where the last inequality is obtained by using Jensen's inequality. This gives the desired result.
\end{proof}
Lemma \ref{lemmmk} and Lemma \ref{lemmnbwue} will be used to derive various
results presented in the sequel.
\begin{lemma}\label{lemmnpsi}
Let $X$ be an absolutely continuous non-negative random variable with
weighted mean inactivity time function $\widetilde{\mu}_\psi(t)$ defined as in (\ref{wmit}).
Assume that there exist non-negative
constants $m$ and $M$ such that $m\leq\phi(t)\leq M$ for all $t\geq 0.$ Then
\begin{equation}\label{wmu}
m\leq \frac{\widetilde{\mu}_\psi(t)}{\widetilde{\mu}(t)}\leq M
\qquad \hbox{for all $t\in D$.}
\end{equation}
\end{lemma}
\begin{proof}
The proof is immediately obtained by recalling (\ref{wmit}) and (\ref{mit}).
\end{proof}
\par
Lemma \ref{lemmnpsi} allows us to obtain ordering relations between the WMIT and MIT functions. Indeed,
(i) if $M=1,$ then $\widetilde{\mu}_\psi(t)\leq \widetilde{\mu}(t)$ for all $t\in D;$
(ii) if $m=1,$ then $\widetilde{\mu}_\psi(t)\geq \widetilde{\mu}(t)$ for all $t\in D.$
\par
For instance, if $\phi(t)=\overline F(t)$, then $M$=1, and $\widetilde{\mu}_\psi(\infty)=\mathbb E[|X-X'|]/2$,
where $X'$ is an independent copy of $X,$ provided that the condition (\ref{EPSIX}) is satisfied.
\par
Hereafter, we focus on a nonparametric class of lifetime distribution based on increasing nature of weighted mean inactivity time function $\widetilde{\mu}_\psi(t).$ As pointed out earlier there is no lifetime distribution with decreasing WMIT function.
\begin{definition}\rm
A non-negative random variable $X$ is said to have increasing weighted mean inactivity time function, denoted by IWMIT, if $\widetilde{\mu}_\psi(t)$ is an increasing function of $t\in D.$
\end{definition}
In the following theorem, we provide sufficient conditions for the increasingness of $\widetilde{\mu}_\psi(t).$ We recall that $X$ is said to be increasing in the mean inactivity time function,  i.e.\ IMIT, if $\widetilde{\mu}(t)$ is increasing in $t$.
\begin{theorem}\label{thm:philambda}
Let $X$ be an absolutely continuous non-negative random variable with reversed hazard rate
function $\tau(x)$ defined as in (\ref{rhr}). If any of the following conditions hold:
\begin{description}
  \item[(i)] ${\phi(x)}/{\tau(x)}$ is an increasing function of $x$;
  \item[(ii)] $\phi(x)$ is increasing in $x$ and $X$ is IMIT;
  \item[(iii)] ${\psi(x)\tau(x)}/{\phi(x)}$ is decreasing in $x>0$;
\end{description}
then $X$ is $IWMIT.$
\end{theorem}
\begin{proof}
The proof under the conditions (i) and (ii) is similar to that the proof of Theorems 1 and 2 of Toomaj and Di Crescenzo \cite{ToomajAD2020Mathematics}, respectively, and hence is omitted.
Now, consider case (iii); since $\psi(t)\geq 0 $ is increasing in $t,$
it is sufficient to prove that the following function is increasing in $t>0$:
\begin{equation}
\frac{\widetilde{\mu}_\psi(t)}{\psi(t)}
=\frac{\displaystyle\int_{0}^{t}\displaystyle \phi(x){F}(x)\,{\rm d}x} {\psi(t)F(t)}
=\frac{\displaystyle\int_{0}^{t}\displaystyle \phi(x){F}(x)\,{\rm d}x}
{\displaystyle\int_{0}^{t}[\psi(x)f(x)+\phi(x){F}(x)]\,{\rm d}x}.
 \label{eq:cexyt}
\end{equation}
Define now
$$
\Psi(i,t):=\int_{0}^{\infty}\nu(i,x)\eta(x,t)\,{\rm d}x,\qquad i=1,2,
$$
where
$$
\nu(i,x)=\left\{
\begin{array}{lcl}
\psi(x)f(x)+\phi(x){F}(x), \ &~~& \ i=1\\[1mm]
\displaystyle \phi(x){F}(x), \ &~~& \ i=2,
\end{array}
\right.
\qquad \hbox{and} \qquad
\eta(x,t)={\bf 1}[x\leq t],
$$
with ${\bf 1}[\cdot]$ the indicator function, i.e.\
${\bf 1}[\pi]=1$ when $\pi$ is true, and ${\bf 1}[\pi]=0$ otherwise. Due to the assumption, $\nu(i,x)$ is $TP_2$ in $(i,x)\in\{1,2\}\times(0,\infty).$
On the other hand, it is easy to observe that $\eta(x,t)$ is $TP_2$ in $(x,t)\in (0,\infty)^2$.
From the general composition theorem of Karlin \cite{Karlin}, it follows that $\Psi(i,t)$
is $TP_2$ in $(i,t)\in\{1,2\}\times (0,\infty)$. This implies that $\widetilde{\mu}_\psi(t)$
is an increasing function of $t$, due to (\ref{eq:cexyt}), and this gives the desired result.
\end{proof}
\begin{remark}\label{rem:xtaux}\rm
\begin{description}
  \item[(i)] It should be noted that the condition that ${\phi(x)}/{\tau(x)}$ is an increasing function of $x$, given in case (i) of Theorem \ref{thm:philambda}, is ensured under the assumptions that $\psi(t)$
is convex and $X$ is DRHR, i.e.\ the reversed hazard rate function $\tau(t)$ is decreasing in $t$.
  \item[(ii)] We point out that if $X$ is an absolutely continuous non-negative random variable
with the reversed hazard rate $\tau(x)$ such that $x\tau(x)$ is decreasing
in $x>0$ and if $\psi(x)=x^r$, $r\geq1$, then
$$
 \frac{\psi(x)\tau(x)}{\phi(x)}=\frac{1}{r}x\tau(x),
$$
is a decreasing function of $x>0.$ In this case, thanks to the Part (iii) of Theorem \ref{thm:philambda}, one can conclude that $X$ is $IWMIT.$ (See Proposition 13 of Di Crescenzo {\em et al.}\ \cite{DMM} for a characterization of the property that $x\tau(x)$ is decreasing).
\end{description}
\end{remark}
\par
The following examples show the usefulness of Theorem \ref{thm:philambda}.
\begin{example}
Let $X$ have Fr\'echet distribution with CDF $F(x)=\exp\{-cx^{-\gamma}\},$ $x>0,$ for $c,\gamma>0.$ Then, under the conditions considered in Part (ii) of Remark \ref{rem:xtaux} we have that $X$ is $IWMIT.$
\end{example}
\begin{example}
Assume that  $\phi(t)=\tau(t)\widetilde{\mu}(t)=1-\widetilde{\mu}'(t)$, where the last equality is due to (\ref{dmutilde}).
From (\ref{weight function}) we thus have $\psi(t)=\int_0^t\phi(u)\,{\rm d}u=t-\widetilde{\mu}(t)$ for all $t>0$.
In this case, from (\ref{wmit}) we get
$$
\widetilde{\mu}_{\psi}(t) =\frac{1}{F(t)}\int_{0}^tf(x)\widetilde{\mu}(x)\,{\rm d}x, \qquad t\in D.
$$
Hence, making use of Theorem 5.2. of Di Crescenzo and Longobardi \cite{Di-Longobardi-2009}, we have
\begin{equation}\label{eq:ce}
\widetilde{\mu}_{\psi}(t)={\cal CE}(X;t),\qquad t>0,
\end{equation}
where ${\cal CE}(X;t)$ is known as the {\em dynamic cumulative entropy} of $X$. Recalling Corollary 6.1 of Di Crescenzo and Longobardi \cite{Di-Longobardi-2009}, we have that if $X$ is IMIT, then  ${\cal CE}(X;t)$ is increasing in $t$, and thus from
(\ref{eq:ce}) we obtain  that $X$ is IWMIT in this case.
This conclusion can also be obtained from point (i) of Theorem \ref{thm:philambda}.
\end{example}
The following example is analogous to Example 2 of Toomaj and Di Crescenzo \cite{ToomajAD2020Mathematics}.
\begin{example}\label{exam:logf}
Let $X$ be an absolutely continuous non-negative random variable with decreasing
and differentiable PDF $f(x)$, with $D=(0,\infty)$ and $0<f(0)<\infty$. Let
$$
 \psi(x)=-\log\frac{f(x)}{f(0)},
 \qquad
 \phi(x)=-\frac{f'(x)}{f(x)}\geq 0, \qquad x>0.
$$
Hence, from (\ref{wmit}) and after some calculations, the WMIT function is given by
\begin{eqnarray}\label{eq:Past}
\widetilde{\mu}_{\psi}(t)
 =\int_0^t  \frac{f(x)}{{F}(t)}\log\frac{f(x)}{{f}(t)}\,{\rm d}x
 =-\overline H(t) -\log \tau(t),
 \qquad t>0,
\end{eqnarray}
where
\begin{equation}
 \overline H(t) =-\int_0^t  \frac{f(x)}{{F}(t)}\log\frac{f(x)}{{F}(t)}\,{\rm d}x,
 \qquad t>0,
 \label{eq:pastentr}
\end{equation}
is the past entropy of $X$ (cf.\ Di Crescenzo and Longobardi \cite{DiCrescenzo2002}
and Muliere {\em et al.}\ \cite{Muliere1993}). In this case, due to condition (i) of
Theorem \ref{thm:philambda}, if
$$
 \frac{f'(x)  F(x)}{f^2(x)} \quad \hbox{is decreasing in $x>0$},
$$
then $X$ is $IWMIT$. Equivalently,  if
$\displaystyle\frac{\tau'(x)}{\tau^2(x)}$ is decreasing in $x>0$ then $X$ is $IWMIT$.
\end{example}
%
\section{Stochastic Comparisons}
In this section, we focus our attention on the relations between weighted the mean inactivity time order
and other some well-known stochastic orders. In this regard, we first recall some stochastic orders (see Shaked and Shanthikumar \cite{Shaked}, and Kayid and Ahmad \cite{KayidAhmad2004}).
\begin{definition}\label{def:stochord}
\rm
Let $X$ and $Y$ be two absolutely continuous non-negative random variables
with cumulative distribution functions ${F}(t)$ and $G(t),$ and mean inactivity time functions
$\widetilde{\mu}_{X}(t)$ and $\widetilde{\mu}_{Y}(t),$ respectively. Then:
\begin{itemize}
  \item $X$ is said to be smaller than $Y$ in the reversed hazard rate order, denoted by $X\leq_{rhr}Y,$ if and only if
$$
 {G}(t)/{F}(t)
 \qquad \hbox{is increasing in $t>0.$}
$$
  \item $X$ is said to be smaller than $Y$ in the mean inactivity time order, denoted by $X\leq_{mit}Y,$ if $\widetilde{\mu}_{X}(t)\geq \widetilde{\mu}_{Y}(t)$ for all $t>0,$ or equivalently,
  $$
 \int_{0}^{t}{G}(x)\,{\rm d}x\Big/\int_{0}^{t}{F}(x)\,{\rm d}x
 \qquad \hbox{is increasing in $t>0.$}
$$
  \item $X$ is said to be smaller than $Y$ in the dispersive order,
denoted by $X\leq_{disp}Y,$ if and only if,
\begin{equation}\label{disp:def}
 G^{-1}(F(t))-t\qquad \hbox{is increasing in $t>0,$}
\end{equation}
where $G^{-1}(u)=\inf\{x\in \mathbb{R}^+:G(x)\geq u\}$, $u\in[0,1],$ denotes the left-continuous quantile function of $G(x).$
\end{itemize}
\end{definition}
\par
Now, we define a new stochastic order in terms of the weighted mean inactivity time function.
\begin{definition}\rm\label{def:WMIT}
For a given non-negative weight function $\phi(\cdot)$, let $X$ and $Y$ have the weighted mean inactivity time functions $\widetilde{\mu}_{\psi(X)}(t)$ and $\widetilde{\mu}_{\psi(Y)}(t),$ respectively. Then, $X$ is said to be smaller than $Y$ in the weighted mean inactivity time function with respect to weight function $\phi(x)$, denoted by $X\leq_{wmit}^\phi Y,$ if and only if,
$$
 \widetilde{\mu}_{\psi(X)}(t)\geq \widetilde{\mu}_{\psi(Y)}(t) \qquad \hbox{for all $t>0$ such that $F(t)>0$ and $G(t)>0$.}
$$
\end{definition}
The next theorem provides equivalent conditions for the weighted mean inactivity time order.
\begin{theorem}\label{equivalent:wmrlthm}
Let $X$ and $Y$ be two lifetime random variables with CDFs $F$ and $G,$ respectively. Then, for any non-negative weight function $\phi(\cdot),$ the following statements are equivalent:
\begin{description}
  \item[(i)] $X\leq_{wmit}^\phi Y;$
  \item[(ii)] $\frac{\int_{0}^{t}\phi(x){G}(x)\,{\rm d}x}{\int_{0}^{t}\phi(x){F}(x)\,{\rm d}x}$ is increasing in $t>0;$
  \item[(iii)] $\mathbb{E}[\psi(X)|X\leq t]\leq \mathbb{E}[\psi(Y)|Y\leq t]$ for all $t>0.$
\end{description}
\end{theorem}
\begin{proof}
In this case, we have
\[
\frac{\rm d}{{\rm d}t}\frac{\int_{0}^{t}\phi(x){G}(x)\,{\rm d}x}
{\int_{0}^{t}\phi(x){F}(x)\,{\rm d}x}
=\frac{\phi(t)\int_{0}^{t}\phi(x)\left[{F}(x){G}(t)-{G}(x){F}(t)\right]{\rm d}x}{\left[\int_{0}^{t}\phi(x){F}(x)\,{\rm d}x\right]^2},\qquad  t>0.
\]
By the definition, one has $X\leq_{wmit}^\phi Y$ if and only if $\int_{0}^{t}\phi(x)\left[{F}(x){G}(t)-{G}(x){F}(t)\right]{\rm d}x\geq0$ for all $t>0.$
This proves that (i) and (ii) are equivalent.
Finally, the equivalence of statements (i) and (iii) is clear from (\ref{wmit}).
\end{proof}
\par
It is worth to mention that the choice $\phi(x)=x$ in Definition \ref{def:WMIT} coincides with the so-called strong mean inactivity time (SMIT) order studied by Kayid and Izadkhah \cite{KayidIzadkhah2014}. As pointed out by the latter authors, the SMIT order lies down between the reversed hazard rate and the MIT orders. So, the WMIT order is a generalization of SMIT order. Hereafter, we will show that the WMIT order can be successfully applied to stochastic ordering between the variance of transformed random variable and the weighted GCE. In the next theorem, we relate the WMIT order to the reversed hazard rate and the mean inactivity time orders.
\begin{theorem}\label{thm:wmit}
Let $X$ and $Y$ be two absolutely continuous non-negative random variables with CDFs ${F}(t)$ and ${G}(t)$
and WMIT functions $\widetilde{\mu}_{\psi(X)}(t)$ and $\widetilde{\mu}_{\psi(Y)}(t),$ respectively,
where  $\psi(\cdot)$ is an increasing non-negative and differentiable function on $(0,\infty).$ Then
\begin{description}
  \item[(i)] if $X\leq_{rhr}Y,$ then $X\leq_{wmit}^\phi Y;$
  \item[(ii)] if $\psi(t)$ is convex on $[0,\infty)$ and $X\leq_{wmit}^\phi Y,$ then $X\leq_{mit}Y.$
\end{description}
\end{theorem}
\begin{proof}
(i) Since $X\leq_{rhr}Y,$ then ${G}(x)/{F}(x)$ is increasing in $x,$ or equivalently
$$
\left[\frac{{F}(x)}{{F}(t)}-\frac{{G}(x)}{{G}(t)}\right]\geq0,\qquad x\leq t.
$$
Since $\phi(x)\geq0,$ from (\ref{wmit}) one can conclude that
$$
 \widetilde{\mu}_{\psi(X)}(t)-\widetilde{\mu}_{\psi(Y)}(t)
 =\int_{0}^{t}\phi(x)\left[\frac{{F}(x)}{{F}(t)}-\frac{{G}(x)}{{G}(t)}\right]{\rm d}x\geq0,
$$
so that $X\leq_{wmit}^\phi Y$.
\\
\noindent (ii) Let
$$
 z_t(x):=\phi(x)\left[\frac{F(x)}{F(t)}-\frac{G(x)}{G(t)}\right]{\bf 1}[x\leq t],
$$
and let ${\rm d}Z_t(x)=z_t(x)\,{\rm d}x.$
Then, for all $t>0$ from (\ref{mit}) we get
$$
\widetilde{\mu}_{Y}(t)-\widetilde{\mu}_{X}(t)
=\int_{0}^{\infty}\frac{1}{\phi(x)} \,{\rm d}Z_t(x)
=\int_{0}^{t}\frac{1}{\phi(x)}\,{\rm d}\left[\int_{0}^{x}\phi(u)\left(\frac{F(u)}{F(t)}-\frac{G(u)}{G(t)}\right){\rm d}u\right],
$$
where $1/\phi(x)$ is a non-negative decreasing function due to assumption.
For all $s>t>0,$ we have
\[
\int_{0}^{s}{\rm d}Z_t(x)=\int_{0}^{t}{\rm d}Z_t(x)=
\int_{0}^{t}\phi(x)\left[\frac{F(x)}{F(t)}-\frac{G(x)}{G(t)}\right]{\rm d}x\geq0,
\]
where the inequality is obtained by the assumption $X\leq_{wmit}^\phi Y.$ Let us assume that $t>s>0.$
Due to (\ref{wmit}), assumption $X\leq_{wmit}^\phi Y$ implies that, for all $t>0,$
\begin{equation}\label{eq:1}
\frac{\int_{0}^{t}\phi(x)F(x)\,{\rm d}x}{\int_{0}^{t}\phi(x)G(x)\,{\rm d}x}\geq\frac{F(t)}{G(t)}.
\end{equation}
In addition, $X\leq_{wmit}^\phi Y$ implies that ${\int_{0}^{x}\phi(u)G(u)\,{\rm d}u}/{\int_{0}^{x}\phi(u)F(u)\,{\rm d}u},$
is increasing in $x,$ and then it holds that, for all $t>s>0,$
\begin{equation}\label{eq:2}
\frac{\int_{0}^{s}\phi(x)F(x)\,{\rm d}x}{\int_{0}^{s}\phi(x)G(x)\,{\rm d}x}\geq
\frac{\int_{0}^{t}\phi(x)F(x)\,{\rm d}x}{\int_{0}^{t}\phi(x)G(x)\,{\rm d}x}.
\end{equation}
Combining (\ref{eq:1}) and (\ref{eq:2}), one gets, for all $t>s>0,$
$$
\frac{\int_{0}^{s}\phi(x)F(x)\,{\rm d}x}{\int_{0}^{s}\phi(x)G(x)\,{\rm d}x}
\geq\frac{F(t)}{G(t)}.
$$
which provides that, for all $t>s>0,$
\[
\int_{0}^{s}{\rm d}Z_t(x)=\int_{0}^{s}\phi(x)\left[\frac{G(x)}{G(t)}
-\frac{F(x)}{F(t)}\right]{\rm d}x\geq0.
\]
Therefore, $X\leq_{wmit}^\phi Y$ implies that $\int_{0}^{s}{\rm d}Z_t(x)\geq 0,$ for all $s,t>0.$ Finally, appealing to Lemma 7.1(b) of Barlow and Proschan \cite{Barlow}, it is concluded that $\int_{0}^{\infty}\frac{1}{\phi(x)}\,{\rm d}Z_t(x)\geq0,$ for all $t>0,$ and hence the proof is completed.
\end{proof}
In the context of Theorem \ref{thm:wmit} (i), Counterexample 1 of Kayid and Izadkhah \cite{KayidIzadkhah2014} shows that $X\leq_{wmit}^\phi Y$ does not imply $X\leq_{rhr} Y$ for an increasing non-negative and differentiable function $\psi(x)={x^2}/{2}.$


\section{Weighted Generalized Cumulative Entropy and Variance}


In the following  two subsections, we discuss  some relevant applications of the weighted mean inactivity time function to supply expressions for the variance of transformed random variable and the weighted generalized cumulative entropy.

\subsection{Variance of Transformed Random Variable}
Recently, Toomaj and Di Crescenzo \cite{ToomajAD2020} showed that the variance of a random variable $X$ can be represented in terms of MIT function as follows:
\begin{equation}\label{sigma2}
 \sigma^2(X)=\mathbb{E}[\widetilde{\mu}^2(X)],
\end{equation}
provided that the expectation exists. In what follows, we extend the result (\ref{sigma2})
to the case of the transformed random variable $\psi(X)$, where $\psi(x)$ is the
cumulative weight function defined in (\ref{weight function}).
Indeed, in the following theorem we express the variance of $\psi(X)$
in terms of the WMIT function (\ref{wmit}).
\begin{theorem}\label{thm:sigma}
Let $X$ be an absolutely continuous non-negative random variable with WMIT function
$\widetilde{\mu}_{\psi}(t),$ and having finite second moment $\mathbb{E}[\psi^2(X)]$. Then
\begin{equation}\label{sigma}
\sigma^2[\psi(X)]=\mathbb{E}[\widetilde{\mu}_{\psi}^2(X)].
\end{equation}
\end{theorem}
\begin{proof}
Let us set
$$
w_{\psi}(x):=\mu_{\psi}(x)F(x)=\int_{0}^{x}\psi(t)f(t)\, {\rm d}t,\quad x>0.
$$
Using (\ref{wmit2}), we obtain
\begin{eqnarray}\label{sig11}
\mathbb{E}[\widetilde{\mu}_{\psi}^2(X)]&=&\int_{0}^{\infty}[\psi(x)-\mu_{\psi}(x)]^2f(x)\, {\rm d}x\nonumber\\
&=&\mathbb{E}(\psi^2(X))+\int_{0}^{\infty}\mu_{\psi}^2(x)f(x)\, {\rm d}x-2\int_{0}^{\infty}\psi(x)\mu_{\psi}(x)f(x)\, {\rm d}x.
\end{eqnarray}
Recalling (\ref{derivemu}), it holds that
\begin{eqnarray*}\label{sig12}
\int_{0}^{\infty}\mu_{\psi}^2(x)f(x)\, {\rm d}x&=&\int_{0}^{\infty}\mu_{\psi}(x)\tau(x)w_{\psi}(x)\, {\rm d}x
\nonumber\\
&=& \int_{0}^{\infty}\psi(x)\tau(x)w_{\psi}(x)\, {\rm d}x -\int_{0}^{\infty}\mu_{\psi}'(x)w_{\psi}(x)\, {\rm d}x
\nonumber\\
&=& \int_{0}^{\infty}\psi(x)\mu_{\psi}(x)f(x)\, {\rm d}x-\int_{0}^{\infty}\mu_{\psi}'(x)w_{\psi}(x)\, {\rm d}x.
\end{eqnarray*}
Integrating by parts gives
$$
\int_{0}^{\infty}\mu_{\psi}'(x)w_{\psi}(x)\, {\rm d}x=[\mathbb{E}(\psi(X))]^2-\int_{0}^{\infty}\psi(x)\mu_{\psi}(x)f(x)\, {\rm d}x,
$$
which implies
\begin{eqnarray}\label{sig13}
\int_{0}^{\infty}\mu_{\psi}^2(x)f(x)\, {\rm d}x&=&2\int_{0}^{\infty}\psi(x)\mu_{\psi}(x)f(x)\, {\rm d}x-[\mathbb{E}(\psi(X))]^2.
\end{eqnarray}
By substituting Eq.\ (\ref{sig13}) into (\ref{sig11}), we have
$$
\mathbb{E}[\widetilde{\mu}_{\psi}^2(X)]=\mathbb{E}[\psi^2(X)]-[\mathbb{E}(\psi(X))]^2=\sigma^2[\psi(X)].
$$
The proof is thus completed.
\end{proof}
\par
We remark that the result expressed in Theorem \ref{thm:sigma} is analogous to the Theorem 3
of Toomaj and Di Crescenzo \cite{ToomajAD2020Mathematics}, where the variance of $\psi(X)$
is expressed as the expectation of the squared weighted mean residual life function of $X$.
\par
In the proof of Theorem \ref{thm:sigma}, we used relation $\psi(0)=0$ due to (\ref{weight function}).
However, by using similar arguments it is not hard to see that for every increasing function $g$, even with
$g(0)\neq 0$, that the variance of $g(X)$ can be expressed as
\[
Var[g(X)]=\mathbb{E}[\widetilde{\mu}_g^2(X)],
\]
where
$$
 \widetilde{\mu}_g(t)=\frac{1}{F(t)}\int_{0}^{t}g'(x){F}(x)\,{\rm d}x,\qquad t>0.
$$
\par
As an application of Eq.\ (\ref{sigma}), let us consider the following example.
\begin{example}
Consider a parallel system composed by $m$ units having    lifetimes $X_1,\ldots,X_m$, which are
i.i.d.\ absolutely continuous random variables with CDF $F(x).$ The system lifetime is thus  $X_{m:m}=\max\{X_1,\ldots,X_m\}$,
whose CDF is given by $F_{m:m}(x):=\mathbb{P}(X_{m:m}\leq x)=[F(x)]^m$, $x\geq 0$.
Setting $\psi(t)=F(t),$ and thus $\phi(t)=f(t)$, from (\ref{wmit}) we obtain, for $t>0$,
\[
 \widetilde{\mu}_{\psi(X_{m:m})}(t)=\frac{1}{F_{m:m}(t)}\int_{0}^{t}f(x)F_{m:m}(x)\,{\rm d}x
 =\frac{1}{[F(t)]^m}\int_{0}^{t}f(x)[F(x)]^m\,{\rm d}x
 =\frac{F(t)}{m+1}.
\]
Thanks to the use of Eq.\ (\ref{sigma2}) and Theorem \ref{thm:sigma}, thus the variance of the
probability integral transformation $F(X_{m:m})$ can be obtained as
\[
\sigma^2[F(X_{m:m})]
 =m\int_{0}^{\infty}f(x)[F(x)]^{m-1}\left[\frac{F(x)}{m+1}\right]^2\,{\rm d}x
 =\frac{m}{(m+1)^2(m+2)}.
\]
\end{example}
\par
Another useful application of Theorem \ref{thm:sigma} involves the so-called varentropy.
If $X$ is absolutely continuous non-negative random variable with PDF $f(x),$ the (random)
information content of $X$ is defined by
$$
 IC(X)=-\log f(X).
$$
It is worth pointing out that $IC(X)$ is the natural counterpart of the number of bits needed to represent $X$
in the discrete case by a coding scheme that minimizes the average code length. It is well known that
\begin{equation}
 H_X=\mathbb{E}[IC(X)]
 =-\int_{0}^{\infty}f(x)\log f(x)\,{\rm d}x
 \label{eq:HXentropy}
\end{equation}
denotes the differential entropy of $X$. The varentropy of $X$ is defined as
(see Di Crescenzo and Paolillo \cite{DiPaol} and references therein)
\begin{eqnarray}
  V(X):= Var(IC(X))&=&\mathbb{E}[(-\log f(X))^2]-[H_X]^2
  \label{eq:defVX}
  \\
  &=&\int_{0}^{\infty}[-\log f(x)]^2f(x)\,{\rm d}x- \left(\int_{0}^{\infty}[-\log f(x)]f(x)\,{\rm d}x\right)^2,
 \nonumber
\end{eqnarray}
so that it measures the variability of the information content of $X$. The relevance of this measure
has been pointed out in various investigations, especially in Fradelizi {\em et al.}\ \cite{Fradelizi}
where an optimal varentropy bound for log-concave distributions is obtained.
\begin{remark}\label{rem:VX}
Let $X$ be an absolutely continuous non-negative random variable with decreasing
and differentiable PDF $f(x)$ over the support $(0,\infty)$ and $0<f(0)<\infty$, and let
$$
 \psi(x)=-\log\frac{f(x)}{f(0)},
 \qquad
 \phi(x)=-\frac{f'(x)}{f(x)}\geq 0, \qquad x>0.
$$
Hence, we have
$$
 V(X)= Var(\psi(X))
 =Var(-\log f(X)),
$$
so that, recalling Example \ref{exam:logf}, due to Eqs.\ (\ref{eq:Past}) and (\ref{sigma}),
we obtain another representation of the varentropy in terms of the past entropy (\ref{eq:pastentr})
and the reversed hazard rate (\ref{rhr})  of $X$ as follows:
$$
V(X)=Var(-\log f(X))=\mathbb{E}\left\{[\overline H(X)+\log \tau(X)]^2\right\}.
$$
On the other hand, recalling Example 2 of
Toomaj and Di Crescenzo \cite{ToomajAD2020Mathematics},
a further expression for the varentropy can be given  as
$$
 V(X)= Var(-\log f(X))=\mathbb{E}\left\{[ H(X)+\log \lambda(X)]^2\right\},
$$
where $\lambda(x)$ is the hazard rate function (\ref{eq:failrate}), and where
$$
 H(t):=-\int_t^{\infty} \frac{f(x)}{\overline F(t)}\log \frac{f(x)}{\overline F(t)}\,{\rm d}x,
 \qquad t>0,
$$
is the residual entropy of $X$, i.e.\ the entropy of the residual lifetime (\ref{eq:residuallifetime}).
\end{remark}
\par
Hereafter, we see that the results stated in Remark \ref{rem:VX} and stimulated by Theorem \ref{thm:sigma}
can be proved under more general assumptions.
\begin{theorem}\label{thm:varentropy}
Let $X$ be an absolutely continuous non-negative random variable with PDF $f(x)$ such that $\mathbb{E}[(IC(X))^2]<\infty$.
Then
\begin{description}
  \item[(i)] $V(X)=\mathbb{E}\left\{[ H(X)+\log \lambda(X)]^2\right\};$
  \item[(ii)] $V(X)=\mathbb{E}\left\{[\overline H(X)+\log \tau(X)]^2\right\}.$
\end{description}
\end{theorem}
\begin{proof}
Let us set
$$
g(x):=-\int_{x}^{\infty}f(z)\log f(z)\, {\rm d}z,\qquad x>0,
$$
such that the differential entropy (\ref{eq:HXentropy}) is given by
$H_X=g(0)$. First note that
\[
H(x)+\log \lambda(x)=\frac{g(x)}{\overline{F}(x)}+\log f(x),\qquad x>0.
\]
Hence, one has
\begin{eqnarray}\label{var11}
\mathbb{E}\left\{[ H(X)+\log \lambda(X)]^2\right\}&=&\int_{0}^{\infty}[\log f(x)]^2f(x)\, {\rm d}x+\int_{0}^{\infty}g^2(x)\frac{f(x)}{\overline{F}^2(x)}\, {\rm d}x\nonumber\\
&+&2\int_{0}^{\infty}f(x)\log f(x)\frac{g(x)}{\overline{F}(x)}\, {\rm d}x.
\end{eqnarray}
By noting that
$$
 g^2(x)=\left(\int_{x}^{\infty}f(z)\log f(z)\, {\rm d}z\right)^2,
$$
and integrating by parts with $u=g^2(x)$ and $v={1}/{\overline{F}(x)},$ we have
\begin{eqnarray}\label{var13}
\int_{0}^{\infty}g^2(x)\frac{f(x)}{\overline{F}^2(x)}\, {\rm d}x&=&\frac{g^2(x)}{\overline{F}(x)}\bigg]_0^{\infty}-2\int_{0}^{\infty}f(x)\log f(x)\frac{g(x)}{\overline{F}(x)}\, {\rm d}x\nonumber\\
&=&-[H_X]^2-2\int_{0}^{\infty}f(x)\log f(x)\frac{g(x)}{\overline{F}(x)}\, {\rm d}x,
\end{eqnarray}
since $\displaystyle\lim_{x\to\infty}\frac{g^2(x)}{\overline{F}(x)}=0$. By substituting Eq.\ (\ref{var13}) into (\ref{var11}), we have
$$
\mathbb{E}\left\{[ H(X)+\log \lambda(X)]^2\right\}=\mathbb{E}[(-\log f(X))^2]-[H_X]^2=Var(-\log f(X)),
$$
where the last equality is due to  (\ref{eq:defVX}).
The proof of Part (i)  is thus completed. The proof of Part (ii) is similar and then is omitted.
\end{proof}
\par
By including a further assumption on $f$, we obtain the following result.
\begin{proposition}
Let the assumptions of Theorem \ref{thm:varentropy} hold.
\\
(i) If $H(t)$ is decreasing in $t$ and
\begin{equation}
\log \frac{f(x)}{f(t)}\leq 1 \quad \hbox{  for all $x\geq t>0$,}
\label{eq:rapplf}
\end{equation}
then $V(X)\leq 1$.
\\
(i) If $\overline H(t)$ is increasing in $t$ and
\begin{equation}
\log \frac{f(x)}{f(t)}\leq 1 \quad \hbox{  for all $0<x\leq t$,}
\label{eq:rapplf2}
\end{equation}
then $V(X)\leq 1$.
\end{proposition}
\begin{proof}
(i) First, we recall that
\[
H(t)+\log \lambda(t)
=-\int_{t}^{\infty}\frac{f(x)}{\overline{F}(t)}\log\frac{f(x)}{f(t)}\,{\rm d}x,
\qquad  t>0.
\]
Hence, by the assumption (\ref{eq:rapplf}) we have $H(t)+\log \lambda(t)\geq -1$, $t>0.$
On the other hand, if $H(t)$ is decreasing in $t,$ then $  H(t)+\log \lambda(t)\leq 1$, $t>0$
(cf.\ Theorem 3.2 of Ebrahimi \cite{Ebrahimi}). Therefore, we get $|H(t)+\log \lambda(t)|\leq 1,$ so that
from Theorem \ref{thm:varentropy} we have $V(X)\leq 1.$ The proof of Part (i)  is thus completed.
In the case (ii), one similarly has
\[
\overline H(t)+\log \tau(t)
=-\int_{0}^{t}\frac{f(x)}{{F}(t)}\log\frac{f(x)}{f(t)}\,{\rm d}x,
\qquad  t>0,
\]
so that from assumption (\ref{eq:rapplf2}) we obtain $\overline H(t)+\log \tau(t)\geq -1$, $t>0.$
Moreover, if $\overline H(t)$ is increasing in $t,$ then $ \overline H(t)+\log \lambda(t)\leq 1$, $t>0$
(cf.\ Proposition 2.3 of Di Crescenzo and Longobardi \cite{DiCrescenzo2002}).
Thus it follows $|\overline H(t)+\log \lambda(t)|\leq 1,$ and finally from Theorem \ref{thm:varentropy} we get $V(X)\leq 1.$
\end{proof}
\par
Clearly, if $f(x)$ is decreasing in $x>0,$ then the condition (\ref{eq:rapplf}) holds. However, such relation can be fulfilled
even for non-decreasing densities. For instance, if $X$ has PDF $f(x)=\frac{1}{3}(1+2 x) e^{-x}$,  $x> 0$, then (\ref{eq:rapplf}) is satisfied.
Moreover, if $f(x)$ is increasing in $x$ on a bounded support, then the condition (\ref{eq:rapplf2}) holds. On the other hand, (\ref{eq:rapplf2}) cannot be fulfilled if $f(t)$ is close to 0, for instance
for large $t$ when $f(x)$ has support $(0,\infty)$. However, relation (\ref{eq:rapplf2})
can be satisfied if $X$ has a bounded support,  for instance when it is uniform on $(a,b),\ a<b$.
\par
In the next theorem, we state that when the weight function is bounded between two  real numbers,
the ratio of standard deviation of transformed random variable with respect to the standard deviation
of the associated random variable also lies down between the same bounds.
\begin{theorem}\label{sigmapsi}
Under the conditions of Lemma \ref{lemmnpsi}, it holds that
$$
m\leq \frac{\sigma[\psi(X)]}{\sigma(X)}\leq M.
$$
In particular, (i) if $m=0$ and $M=1,$ then $\sigma[\psi(X)]\leq \sigma(X)$ and,
(ii) if $m=1$ and $M<\infty,$ then $\sigma[\psi(X)]\geq \sigma(X)$.
\end{theorem}
\begin{proof}
The proof is immediately obtained from (\ref{wmu}) and recalling (\ref{sigma2}) and (\ref{sigma}).
\end{proof}
Now, let us consider two applications in the following examples.
\begin{example}
Let $X$ and $Y$ be   non-negative random lifetimes with CDFs $F$ and $G,$ respectively.
Consider the function $\psi(t)=G^{-1}F(t)$, which is increasing in $t>0.$  Due to (\ref{weight function}), we have
that $\phi(x)\geq1$ if and only if $\psi(t)-t$ is increasing in $t.$ Supposing that $X\leq_{disp}Y,$ one can conclude that $\psi(t)-t=G^{-1}F(t)-t$ is increasing $t$ by recalling (\ref{disp:def}). Making use of Theorem \ref{sigmapsi}, we have
\[
\sigma[\psi(X)]=\sigma(G^{-1}F(X))\geq \sigma(X).
\]
By noting that $G^{-1}F(X)\stackrel{d}{=}Y,$ where $\stackrel{d}{=}$ means equality in distribution,
we obtain the well-known result $\sigma(X)\leq \sigma(Y)$.
\end{example}
\begin{example}
Assume that $X_1,X_2,\ldots,X_n$ are independent and identically distributed random lifetimes with the common CDF $F(x)$ and PDF $f(x)$.
The $i$th smallest   value is usually called the $i$th order statistic, and is denoted by $X_{i:n}$, $i=1,2,\ldots,n$. Let us set $\psi(x)=F(x)$ and thus $\phi(x)=f(x).$ If $S$ is the support of $f$, then
$$
 \inf_{x\in S}f(x) =:m\leq f(x)\leq M:=\sup_{x\in S}f(x).
$$
It is known that the probability integral transform $V_i=F(X_{i:n})$  has a beta distribution with parameters $i$ and $n-i+1,$ respectively.
Since
$$
 \sigma^2[V_i]=\sigma^2[F(X_{i:n})]=\frac{i(n-i+1)}{(n+1)^2(n+2)},
 \qquad i=1,2,\ldots,n,
$$
from Theorem \ref{sigmapsi} we have
\[
\frac{i(n-i+1)}{M^2(n+1)^2(n+2)}\leq \sigma^2[X_{i:n}]\leq \frac{i(n-i+1)}{m^2(n+1)^2(n+2)},\qquad i=1,2,\ldots,n
\]
provided that $0<m\leq M<\infty$.
Specifically, after some simplifications the average variance of the order statistics is bounded as follows:
\[
\frac{1}{6M^2(n+1)}\leq \frac{1}{n}\sum_{i=1}^{n}\sigma^2[X_{i:n}]\leq \frac{1}{6m^2(n+1)}.
\]
The latter result is useful to show that when $n$ goes to infinity, then the average variance of the order statistics vanishes, i.e.
\[
\lim_{n\to\infty}\frac{1}{n}\sum_{i=1}^{n}\sigma^2[X_{i:n}]
=\lim_{n\to\infty}\frac{1}{n}\sum_{i=1}^{n}\sigma^2[X_i]=0,
\]
provided that $0<m\leq M<\infty$.
\end{example}
\par
In the next theorem, we provide a connection between the variance of the weighted random variable $\psi(X)$
and the cumulative entropy. For a non-negative random variable $X$ with CDF $F(x)$ and support $(0,\infty)$,
the cumulative entropy (CE), defined by (see Di Crescenzo and Longobardi \cite{DiLongobardi2006})
\begin{equation}\label{ce}
\mathcal{CE}(X)=-\int_{0}^{\infty}F(x)\log{F}(x)\,{\rm d}x
=\int_{0}^{\infty}F(x)\,T(x)\,{\rm d}x,
\end{equation}
where
\begin{equation}\label{chf}
T(x)=-\log F(x)=\int_{x}^{\infty}\tau(u)\,{\rm d}u,\qquad x>0,
\end{equation}
denotes the cumulative reversed hazard function. Another useful representation of $\mathcal{CE}(X)$ is given
in terms of the MIT function as follows:
$$
\mathcal{CE}(X)=\mathbb{E}[\widetilde{\mu}(X)]=\int_{0}^{\infty}\widetilde{\mu}(x)f(x)\,{\rm d}x.
$$
Several properties of CE of (\ref{ce}) as well as its dynamic version are widely discussed in Di Crescenzo and Longobardi \cite{DiLongobardi2006} and Navarro \emph{et al.} \cite{Navarro-2010} and references therein.
\begin{theorem}\label{thm:CESigma}
If $\psi(x)$ is an increasing convex and differentiable function, then,
$$
\sigma[\psi(X)]\geq\psi(\mathcal{CE}(X)).
$$
\end{theorem}
\begin{proof}
The proof is similar to that  of Point (ii) of Theorem 2 of Toomaj and Di Crescenzo \cite{ToomajAD2020}.
\end{proof}
%
\subsection{Weighted Generalized Cumulative Entropy}
%
As noted in (\ref{eq:HXentropy}), for an absolutely continuous non-negative random variable $X$ having PDF $f,$ the differential entropy is given by
$H_X=-\mathbb{E}[\log f(X)].$ It assigns equal importance (or weights) to the occurrence of every event of the form $\{X=x\}$.
However, in certain situations they  have different qualitative characteristic usually known as utility of an outcome. This
motivated  to define the {\em weighted entropy\/} of $X$ as (cf.\ Di Crescenzo and Longobardi \cite{DiLongobardi2006})
\begin{equation}\label{weighted entropy}
H^w(X)=-\mathbb{E}[X\log f(X)]=-\int_{0}^{\infty}xf(x)\log f(x)\,{\rm d}x.
\end{equation}
In analogy with (\ref{weighted entropy}), Misagh \emph{et al.}\ \cite{Misagh2011} proposed an alternative weighted measure called
{\em weighted cumulative entropy\/} (WCE) and based on the distribution function $F(x)$ instead of the PDF $f(x)$ in (\ref{weighted entropy}), defined by
\begin{equation}\label{wcre}
 \mathcal{CE}^w(X)=-\int_{0}^{\infty}xF(x)\log F(x)\,{\rm d}x
 =\int_{0}^{\infty}x{F}(x)T(x)\,{\rm d}x,
\end{equation}
with $T(x)$ defined in (\ref{chf}).
Recently, the WCE was extended by Tahmasebi \emph{et al.}\  \cite{Tahmasebietal2019} to the weighted
generalized cumulative entropy (WGCE) given by
\begin{eqnarray}
 \mathcal{CE}^{\phi}_n(X)=\int_{0}^{\infty}\phi(x)\frac{T^n(x)}{n!}F(x) \, {\rm d}x,
 \label{TEWCE}
\end{eqnarray}
for all $n\in\mathbb{N}:=\{1,2,\ldots\}$, and for any non-negative weight function $\phi(x)$. In particular by taking  $\phi(x)\equiv 1$ in (\ref{TEWCE}), we immediately derive the generalized cumulative entropy (GCE) introduced by Kayal \cite{Kayal}. Several results on weighted entropies are investigated and discussed in Mirali and Baratpour \cite{Mirali2017b}, Misagh \emph{et al.}\ \cite{Misagh2011}, Suhov and Yasaei Sekeh \cite{Yasaei} and Tahmasebi \cite{Tahmasebietal2019}.
Despite the various investigations of these measures, the analysis of their exact meaning and interpretation can still be improved.
\par
Suppose that $\{Y_n, n\in \mathbb N\}$ is a sequence of non-negative i.i.d.\ random variables having the common CDF $F(x)$.
We say that $Y_i$ is a lower record value of this sequence if $Y_i<\min\{Y_1,Y_2,\ldots,Y_{i-1}\}$,
with $i>1$, and by definition $Y_1$ is a lower record value. Let $L(1)=1$ and $L(n+1)=\min\{j: j>L(n), Y_j<Y_{L(n)}\}$ for $n\in\mathbb{N}$, so that $L(n)$ denotes the index where the $n$th lower record value occurs. The random variables $X_{n+1}=Y_{L(n+1)}$, $n\in \mathbb{N}_0:=\{0,1,\ldots\} $, are said to be the lower records, such that $Y_{L(1)}\stackrel{d}{=}X.$ Denoting by $F_{n+1}(x)$ the cumulative distribution function of $X_{n+1},$ $n\in\mathbb{N}_0,$ it follows that
\begin{equation}\label{cdfYn+1}
     F_{n+1}(x) = F(x) \, \sum_{k=0}^{n} \frac{T^k(x)}{k!},
     \qquad x \geq 0,
\end{equation}
so that the PDF of $X_{n+1}$ is given by
\begin{equation}\label{pdfrecord}
 f_{n+1}(x)
 =f(x)\frac{T^n(x)}{n!}, \qquad x \geq 0,
\end{equation}
where $T(x)$ is the cumulative reversed hazard function defined in (\ref{chf}). We recall that the GCE of order $n$ of $X$ is given by (see Kayal \cite{Kayal}, and Di Crescenzo and Toomaj \cite{Di-Toomaj})
\begin{equation}\label{gce}
 \mathcal{CE}_n(X)=\int_{0}^{\infty}F(x)\frac{T^n(x)}{n!}\,{\rm d}x
 =\int_{0}^{\infty}F(x)\frac{[-\log{F}(x)]^n}{n!}\,{\rm d}x,
\end{equation}
for all $n\in\mathbb{N}$. Thus, the GCE of order $n$ corresponds to the expected spacings of lower record values. Let us now provide a suitable extension of $\mathcal{CE}_n(X).$ For all increasing non-negative and differentiable function $\psi(x),$ the weighted GCE of $X$ is expressed as follows:
\begin{eqnarray}
\mathcal{CE}_{\psi,n}(X) \!\!\!\! &=& \!\!\!\! \mathbb{E}[\psi(X_{n})-\psi(X_{n+1})]
 =   \int_0^{\infty}\phi(x)\left[F_{n+1}(x)- F_{n}(x)\right] {\rm d}x,
 \nonumber
 \\
 &=& \!\!\!\! \int_{0}^{\infty}\phi(x)\frac{T^{n}(x)}{n!}F(x) \, {\rm d}x=
 \mathbb{E}\left[\frac{\phi(X_{n+1})}{\tau(X_{n+1})}\right],
 \qquad n\in\mathbb{N}.
 \label{eq:EnXd2}
\end{eqnarray}
Note that for $n = 0,\ \mathcal{CE}_{\phi,0}(X)=\int_{0}^{\infty}\psi(x)F(x)\,{\rm d}x,$ which may be divergent. Hence, the function $\mathcal{CE}_{\psi,n}(X)$ can be identified with the WGCE introduced in (\ref{TEWCE}). This measure extends the GCE through a suitable $\psi.$ For example, if we take $\psi(t)=t,$ then the WGCE coincides with the GCE introduced by Kayal \cite{Kayal}, see also Di Crescenzo and Toomaj \cite{Di-Toomaj} and Toomaj and Di Crescenzo \cite{ToomajAD2020}. Moreover, if we take $\psi(t)=\frac{t^2}{2},$ it concurs with the weighted GCE introduced by Kayal and Moharana \cite{KayalMoharana}.
We note that $\mathcal{CE}_{\psi,n}(X)$ can be viewed as the area between the functions ${F}_{\psi(X_{n})}(x)$ and ${F}_{\psi(X_{n+1})}(x),$ since
from (\ref{eq:EnXd2}) we have
\begin{eqnarray*}
\mathcal{CE}_{\psi,n}(X)= \mathbb{E}[\psi(X_{n})-\psi(X_{n+1})]
= \int_0^{\infty}\left[F_{\psi(X_{n+1})}(x)- F_{\psi(X_{n})}(x)\right]{\rm d}x,
 \qquad n\in\mathbb{N}.
\end{eqnarray*}
Proceeding similarly as in the proof of Proposition 1  of Toomaj and Di Crescenzo \cite{ToomajAD2020}, from (\ref{eq:EnXd2})
one can see that the weighted GCE of $X$ is equivalent to the GCE of a cumulative weighted random variable $\psi(X)$, i.e.\ $\mathcal{CE}_{\psi,n}(X)=\mathcal{CE}_{n}(\psi(X))$ for all $n\in\mathbb{N}.$
\par
With reference to the GCE, defined in (\ref{gce}), in the following theorem we obtain a result analogous to Theorem 7 of Toomaj and Di Crescenzo \cite{ToomajAD2020Mathematics}. The proof is omitted, being similar to that theorem by virtue of the following relation
$$
 \int_{t}^{\infty}\frac{T^{n-1}(x)}{(n-1)!}\,\tau(x)\,{\rm d}x
 =\frac{T^n(t)}{n!},\qquad  t\geq 0, \;\; n\in\mathbb{N}.
$$
\begin{theorem}\label{th:rit}
Let $X$ be an absolutely continuous non-negative random variable with 
weighted mean inactivity time  function $\widetilde{\mu}_{\psi}(t)$. Then, for all $n\in\mathbb{N}$ one has
\begin{equation}\label{eq:rit}
\mathcal{CE}_{\psi,n}(X)=\mathbb{E}[\widetilde{\mu}_{\psi}(X_n)].
\end{equation}
\end{theorem}
In the following theorem, we determine two recurrent expressions for the GCE analogous to those given in Theorems 4 and 5
of Toomaj and Di Crescenzo \cite{ToomajAD2020} and thus the proof is omitted.
\begin{theorem}\label{thm:iterative}
Under the assumption of Theorem \ref{th:rit}, for all $n\in\mathbb{N},$ we have\\
{\bf (i)}
\[
\mathcal{CE}_{\psi,n}(X)=\mathcal{CE}_{\psi,n-1}(X)-\frac{1}{(n-1)!}\,\mathbb{E}[\widetilde{h}_{\psi,n}(X)],
\]
where
\[
\widetilde{h}_{\psi,n}(t):=\int_{0}^{t}\widetilde{\mu}_{\psi}'(x)\,T^{n-1}(x)\,{\rm d}x.
\]\\
{\bf (ii)}
\[
\mathcal{CE}_{\psi,n}(X)=\mathcal{CE}_{\psi,n-1}(X)\left\{1-\mathbb{E}[\widetilde{\mu}_{\psi}'(\widetilde{Z})]\right\},
\]
where $\widetilde{Z}$ is an absolutely continuous non-negative random variable having PDF
\[
 f_{\widetilde{Z}}(x)=\frac{F(x)}{\mathcal{CE}_{\psi,n-1}(X)}\,\frac{T^{n-1}(x)}{(n-1)!},\qquad  x>0.
\]
\end{theorem}
It is worth to mention that when $X$ is IWMIT, since $\widetilde{\mu}'_\psi(x)\geq0,$ as an immediate consequence of Theorem \ref{thm:iterative} we have
$$
\mathcal{CE}_{\psi,n}(X)\geq\mathcal{CE}_{\psi,n-1}(X),
\qquad \hbox{for all $n\in\mathbb{N}.$}
$$
Hereafter we obtain an upper bound for the WGCE in terms of the expected value of the squared weighted mean inactivity time. The proof is omitted being similar to Theorem 6 of Toomaj and Di Crescenzo \cite{ToomajAD2020}.
\begin{theorem} 
Let $X$ be an absolutely continuous non-negative random variable and let $\psi(x)$ denote a non-negative weight function. Then, for all $n\in\mathbb{N}$,
$$
\mathcal{CE}_{\psi,n}(X)\leq \frac{\sqrt{[2(n-1)]!}}{(n-1)!}\,\sigma[\psi(X)].
$$
\end{theorem}
\begin{remark}
We note that, due to Remark 6 of Toomaj and Di Crescenzo \cite{ToomajAD2020Mathematics}, we have
$H[\psi(X)]
=H(X)+\mathbb{E}[\log\phi(X)].$
Hence, by making use of Remark 6 of Toomaj and Di Crescenzo \cite{ToomajAD2020Mathematics}
and Proposition 5 of Tahmasebi \emph{et al.}\ \cite{Tahmasebietal2019},
the following lower bound can be immediately obtained:
$$
\mathcal{CE}_{\psi,n}(X)\geq \frac{1}{n!} C_n \exp{ \{ H(\psi(X))\} },
\qquad n\in\mathbb{N},
$$
where $C_n=\exp\{\int_{0}^{1}\log (u[-\log u]^n)\,{\rm d}u\}$.
\end{remark}
\par
Further useful results are given below.
\begin{theorem}\label{wgcepsi}
Let $X$ be an absolutely continuous non-negative random variable and let $\psi(x)$ denote a non-negative weight function. Let $n\in\mathbb{N}.$\\
(i) If $\psi(x)$ is an increasing convex (concave) function on $(0,\infty),$ then
\begin{itemize}
  \item $\displaystyle\frac{\mathcal{CE}_{\psi,n}(X)}{\mathcal{CE}_n(X)}$ is decreasing (increasing) in $n\in\mathbb{N}$;
  \item $\mathcal{CE}_{\psi,n}(X)\geq(\leq)\, \psi(\mathcal{CE}_n(X))$\ \hbox{for all $n\in\mathbb{N}.$}
\end{itemize}
(ii) Under the condition of Lemma \ref{lemmnpsi}, it holds that
$$
m\leq \frac{\mathcal{CE}_{\psi,n}(X)}{\mathcal{CE}_n(X)}\leq M.
$$
In particular,   if  $M=1$ then $\mathcal{CE}_{\psi,n}(X)\leq \mathcal{CE}_{n}(X)$, whereas
if $m=1$ then $\mathcal{CE}_{\psi,n}(X)\geq \mathcal{CE}_{n}(X).$
\end{theorem}
\begin{proof}
(i)
The proofs are analogue to Theorems 8 and 10 of Toomaj and Di Crescenzo \cite{ToomajAD2020Mathematics}, respectively.
(ii) The proof is immediately obtained from (\ref{wmu}) and recalling   (\ref{eq:rit}).
\end{proof}

In the next theorem,  we provide different probabilistic expressions for the WGCE.
The second one involves the covariance of the transformation of the $n$-th epoch time and the random variable $T(X_n).$ The proof is similar to that of Theorem 13 of Toomaj and Di Crescenzo \cite{ToomajAD2020Mathematics},
and thus is omitted.
\begin{theorem}\label{th:ECov}
For all $n\in\mathbb{N},$ it holds that
\begin{description}
  \item[(i)]
  $\displaystyle\frac{1}{n}\,\mathbb{E}
 \bigg[\displaystyle\frac{\phi(X_n)T(X_n)}{\tau(X_n)}\,\bigg]
  =\mathcal{CE}_{\psi,n}(X);$
  \item[(ii)]
  $\displaystyle\frac{1}{n}\, {\rm Cov}[\psi(X_n),T(X_n)]
  = -\,\mathcal{CE}_{\psi,n}(X).$
\end{description}
\end{theorem}
We can now prove the following  theorem, which allows to compare the WGCE of two random variables under the dispersive ordering.
\begin{theorem}
Let $X$ and $Y$ be  absolutely continuous non-negative random variables, and let $\psi$ be a
cumulative weight function  defined as in (\ref{weight function}).
If $\psi(X)\leq_{disp} \psi(Y),$ then $\mathcal{CE}_{\psi,n}(X)\leq \mathcal{CE}_{\psi,n}(Y)$ for all $n\in \mathbb{N}.$
\end{theorem}
\begin{proof}
Let us consider the cumulative weighted random variables $\psi(X)$ and $\psi(Y)$ with CDFs $H$ and $Q,$ respectively. It is easy to see that
\begin{equation}\label{stackXY}
\psi(Y)\stackrel{d}{=}Q^{-1}H(\psi(X)),
\end{equation}
where $Q^{-1}H$ is an increasing function. Since $\psi(X)\leq_{disp} \psi(Y),$ by the Definition \ref{def:stochord} it holds that
$Q^{-1}H(x)-x$ is increasing in $x>0.$ Taking into account that $\mathcal{CE}_{\varphi,n}(X)=\mathcal{CE}_{n}(\varphi(X))$ for an increasing function $\varphi,$ by taking $\varphi(x)=Q^{-1}H(x)$, Part (ii) of Theorem \ref{wgcepsi} implies that $\mathcal{CE}_{n}(Q^{-1}H(\psi(X)))\geq \mathcal{CE}_{n}(\psi(X))$ for all $n\in \mathbb{N}.$
From (\ref{stackXY}), we immediately obtain that $\mathcal{CE}_{n}(\psi(Y))\geq \mathcal{CE}_{n}(\psi(X))$,
which yields $\mathcal{CE}_{\psi,n}(Y)\geq \mathcal{CE}_{\psi,n}(X)$ for all $n\in \mathbb{N}.$
\end{proof}
In the following theorem, we can show that if two random variables $X$ and $Y$ are ordered with respect to their reversed failure rate functions, then their corresponding variance and WGCE will be ordered too, provided that a weighted MIT is increasing, and the cumulative weight functions are increasing. We recall that if $X$ is greater than $Y$ in the usual stochastic order, denoted by $X\geq_{st}Y,$ then
\begin{equation}\label{stmonotone}
\mathbb{E}[h(X)]\geq\, \mathbb{E}[h(Y)],
\end{equation}
for all increasing functions $h.$
\begin{theorem}\label{reversedharazd:thm}
Let $X$ and $Y$ be  absolutely continuous non-negative random variables with 
weighted mean inactivity time functions $\widetilde{\mu}_{\psi(X)}(t)$ and $\widetilde{\mu}_{\psi(Y)}(t),$ respectively, such that $X\geq_{st}Y.$ If $X\leq_{wmit}^\phi Y$ and either $X$ or $Y$ is IWMIT, then
\begin{description}
  \item[(i)] $\sigma^2[\psi(X)]\geq\sigma^2[\psi(Y)];$
  \item[(ii)] $\mathcal{CE}_{\psi,n}(X)\geq \mathcal{CE}_{\psi,n}(Y)$, for all $n\in \mathbb{N}.$
\end{description}
\end{theorem}
\begin{proof}
(i) Let $X$ be IWMIT. From (\ref{sigma}), we get
\begin{eqnarray*}
  \sigma^2[\psi(X)]=\mathbb{E}[\widetilde{\mu}_{\psi(X)}^2(X)]
  \geq\mathbb{E}[\widetilde{\mu}_{\psi(X)}^2(Y)]\geq \mathbb{E}[\widetilde{\mu}_{\psi(Y)}^2(Y)]
  =\sigma^2[\psi(Y)].
\end{eqnarray*}
The first inequality is obtained by noting that $X$ is IWMIT, so that $\widetilde{\mu}_{\psi(X)}^2(t)$ is increasing, and by virtue of (\ref{stmonotone}). The last inequality is obtained by the fact that $X\leq_{wmit}^\phi Y$ implies
$\widetilde{\mu}_{\psi(X)}(t)\geq \widetilde{\mu}_{\psi(Y)}(t),\ t>0,$ due to Definition \ref{def:WMIT}. When $Y$ is IWMIT, the proof is similar.
\\
(ii) Let $X$ be IWMIT. From Theorem \ref{th:rit}, for all $n\in\mathbb N,$ we get
\begin{eqnarray*}
  \mathcal{CE}_{\psi,n}(X)=\mathbb{E}[\widetilde{\mu}_{\psi(X)}(X_n)]\geq
  \mathbb{E}[\widetilde{\mu}_{\psi(X)}(Y_n)]\geq \mathbb{E}[\widetilde{\mu}_{\psi(Y)}(Y_n)]=\mathcal{CE}_{\psi,n}(Y).
\end{eqnarray*}
The first inequality is obtained as follows: It is not hard to find that $X\geq_{st}Y$ implies $X_n\geq_{st}Y_n$ for all $n\in \mathbb{N},$ and hence the first inequality is concluded by virtue of (\ref{stmonotone}) since $\widetilde{\mu}_{\psi(X)}(t)$ is increasing. The second inequality is obtained noting that assumption $X\leq_{wmit}^\phi Y$ implies $\widetilde{\mu}_{\psi(X)}(t)\geq \widetilde{\mu}_{\psi(Y)}(t),\ t>0$, from Definition \ref{def:WMIT}. When $Y$ is IWMIT, the proof is similar.
\end{proof}
%

\section{Connection with the location-independent riskier order}

In the last decades, the attention of scholars on quantiles of probability distributions has increased continuously, since they have
an immediate interpretation in terms of over/or undershoot probabilities. Several applications of quantiles have been oriented
to current problems of risk management involving the concept of value-at-risk (VaR). For a random variable $X$ with CDF $F,$ the VaR or left-continuous inverse (quantile function) is defined by
\[
F^{-1}(p)=\inf\{x \in \mathbb R : F(x)\geq p\},\quad \hbox{for}\ p\in(0,1).
\]
In today's financial world, VaR has become the benchmark risk measure: its importance is unquestionable since regulators
accept this model as the basis for setting capital requirements for market risk exposure; see e.g.\ Denuit \emph{et al.} \cite{Denuit}.
The excess wealth transform (or right spread function) of a random variable $X$ with distribution function $F$ and with a finite mean, is defined by
\begin{equation}\label{excess}
W_X(p)=\mathbb{E}[(X-F^{-1}(p))^+]=\int_{F^{-1}(p)}^{\infty}\overline{F}(x)\,{\rm d}x
=\int_{p}^{1}(F^{-1}(q)-F^{-1}(p))\,{\rm d}q,
\end{equation}
for $p\in(0,1).$ For any real number $a,$ we denote by $a^+$ its positive part, that is, $a^+=a$ if $a>0,$ and $a^+=0$ if $a\leq 0.$
We remark that it is not necessary for the random variable $X$ to be non-negative in order for $W_X(p)$ to be well defined.
Indeed, it is only required that $X$ has a finite mean.
Based on this concept, the excess wealth order (or the right spread order) is introduced in Fern\'andez-Ponce \emph{et al.} \cite{Fernandez-Ponce1998},
by expressing that the expected shortfall risk measure (for the positive tail) is comparable, that is, $\mathbb{E}[(X-F^{-1}(p))^+]\leq \mathbb{E}[(Y-G^{-1}(p))^+]$ for all $p\in (0,1).$
Some applications of this function and the excess wealth order are considered in
Toomaj and Di Crescenzo \cite{ToomajAD2020} and \cite{ToomajAD2020Mathematics}.
Hereafter, we define the left spread function, which is  dual to the right spread function given in (\ref{excess}).
\begin{definition}\label{def:4}
Let $X$ be a random variable having CDF $F(x)$ and with finite mean. The left spread function of $X$, 
for $0<p<1$ is defined by
$$
\widetilde{W}_X(p)=\mathbb{E}[(F^{-1}(p)-X)^+]=\int_{0}^{F^{-1}(p)}F(x)\,{\rm d}x
=\int_{0}^{p}(F^{-1}(p)-F^{-1}(q))\,{\rm d}q.
$$
\end{definition}
\par
The left spread function is an increasing function of $p$. Moreover, it is closely related to the MIT function given in (\ref{mit}) by the following relation, if $X$ is non-negative:
$$\displaystyle\widetilde{\mu}(F^{-1}(p))=\mathbb{E}[F^{-1}(p)-X|X\leq F^{-1}(p)]=\frac{\widetilde{W}_X(p)}{p},
\quad 0<p<1.$$
Thanks to the previous identity, 
in the next theorem we show that the variance and the GCE of a random variable can be expressed in terms of the left spread function.
The results follow from Theorems 19 and 21 of Toomaj and Di Crescenzo \cite{ToomajAD2020} and thus the proof is omitted.
\begin{theorem}\label{thm:vargce}
Let $X$ denote an absolutely continuous non-negative random variable with CDF $F.$ Then it holds that
\begin{description}
  \item[(i)] $\displaystyle Var(X)=\int_{0}^{1}\left[\widetilde{\mu}(F^{-1}(p))\right]^2\,{\rm d}p,$
  \item[(ii)] $\displaystyle\mathcal{CE}_{n}(X)=\frac{1}{(n-1)!}\int_{0}^{1}\widetilde{\mu}(F^{-1}(p))(-\log p)^{n-1}\,{\rm d}p$,\quad for all $n\in\mathbb{N}.$
\end{description}
\end{theorem}
Let us consider the following example.
\begin{example}
If $X$ is uniformly distributed in $[0,b],$ then
\[
\widetilde{\mu}(F^{-1}(p))=\frac{bp}{2}.
\]
Recalling Theorem \ref{thm:vargce}, we get
\[
Var(X)=\int_{0}^{1}\left[\widetilde{\mu}(F^{-1}(p))\right]^2{\rm d}p=\frac{b^2}{12}.
\]
On the other hand, for any $n\in\mathbb{N}$ we obtain
$$
\mathcal{CE}_n(X)=\frac{b}{2(n-1)!}\int_{0}^{1}p\,(-\log p)^{n-1}\,{\rm d}p
=\frac{b}{2^{n+1}}.
$$
\end{example}
In economics, many stochastic orders are built to compare the risks of two random assets.
To keep the comparison independent of locations, Jewitt \cite{Jewitt} proposes the following concept.
A non-negative random asset $Y$ is said to be location independent riskier than another non-negative random asset $X,$ denoted by $X\leq_{lir}Y,$ if and only if,
\[
 \int_{0}^{F^{-1}(p)} F(x)\,{\rm d}x\leq  \int_{0}^{G^{-1}(p)} G(x)\,{\rm d}x,\quad \hbox{for all $p\in(0,1)$},
\]
or equivalently
\begin{equation}\label{eq:qmeaninactivity}
\widetilde{\mu}_X(F^{-1}(p))\leq \widetilde{\mu}_Y(G^{-1}(p)),\quad \hbox{for all $p\in(0,1)$},
\end{equation}
where $\widetilde{\mu}_X$ and $\widetilde{\mu}_Y$ denote the MIT functions of $X$ and $Y$, respectively. Roughly speaking, if the inequality (\ref{eq:qmeaninactivity}) holds then
$Y$ has more weight in the lower tail than $X$.
Intuitively, having a great weight in the lower tail is something which should be avoided by risk averters. One advantage of the above definition is that it is a ``choice based" criterion of risk which does not stipulate that the distributions have equal means. The proof of the next lemma is straightforward due to Theorem \ref{thm:vargce} and applying (\ref{eq:qmeaninactivity}), and therefore it is omitted.
\begin{theorem}\label{thm:lir}
Let $X$ and $Y$ be two absolutely continuous non-negative random variables such that $X\leq_{lir}Y.$ Then
\begin{description}
  \item[(i)] $Var(X)\leq Var(Y);$
  \item[(ii)] $\mathcal{CE}_{n}(X)\leq \mathcal{CE}_{n}(Y)$ for all $n\in\mathbb{N}.$
\end{description}
\end{theorem}
Based on Theorem \ref{thm:lir}, it is worth pointing out that if $Y$ is more risky than $X$ both in the  variance and GCE, then it has a larger variance and GCE. Hereafter, we obtain expressions for the transformed random variable and weighted GCE in terms of transformed excess wealth function. For an absolutely continuous non-negative random variable $X$ with CDF $F(x),$ assume that $\psi(\cdot)$ is an increasing non-negative function defined by (\ref{weight function}). The {\em transformed (or weighted) left spread function}, for all $0<p<1,$ is defined by
\begin{eqnarray}\label{trasexcesswealth}
\widetilde{W}_{\psi(X)}(p)  &=&   \mathbb{E}[(\psi(F^{-1}(p))-\psi(X))^+]
=\int_{0}^{F^{-1}(p)} \phi(x)F(x)\,{\rm d}x
\nonumber\\
&=&   \int_{0}^{p}\left[\psi(F^{-1}(p))-\psi(F^{-1}(q))\right]\,{\rm d}q.
\end{eqnarray}
When $\psi(t)=t,$ then 
from (\ref{trasexcesswealth}) we have that $\widetilde{W}_{\psi(X)}(p)$ is equal to the 
left spread function introduced in Definition \ref{def:4}. Moreover,  
this function is related to the weighted mean inactivity time function by the following relation:
\begin{equation}\label{mrlexcesstrans}
\widetilde{\mu}_{\psi(X)}(F^{-1}(p))=\frac{\widetilde{W}_{\psi(X)}(p)}{p},\qquad 0<p<1.
\end{equation}
Now, in the following theorem, we provide expressions for both the variance of a transformed random variable and the weighted GCE in terms of (\ref{mrlexcesstrans}).
\begin{theorem}\label{thm:vargcre}
Let $X$ denote an absolutely continuous random variable with CDF $F.$ Then it holds that
\begin{description}
  \item[(i)] $\displaystyle Var[\psi(X)]=\int_{0}^{1}\left[\widetilde{\mu}_{\psi(X)}(F^{-1}(p))\right]^2{\rm d}p,$
  \item[(ii)] $\displaystyle\mathcal{CE}_{\psi,n}(X)=\frac{1}{(n-1)!}\int_{0}^{1}\widetilde{\mu}_{\psi(X)}(F^{-1}(p))(-\log p)^{n-1}\,{\rm d}p$,\quad for all $n\in\mathbb{N}.$
\end{description}
\end{theorem}
\begin{proof}
(i) By taking  $p=F(x),$ it holds that
\begin{eqnarray*}
  \int_{0}^{1}\left[\widetilde{\mu}_{\psi(X)}(F^{-1}(p))\right]^2{\rm d}p&=&\int_{0}^{\infty}\left[\widetilde{\mu}_{\psi(X)}(x)\right]^2{\rm d}F(x) 
  =Var[\psi(X)],
\end{eqnarray*}
where the last equality is obtained from Theorem \ref{thm:sigma}. The proof of Part (i) is thus completed.
By virtue of (\ref{eq:rit}), Part (ii) can be proved in a similar way.
\end{proof}


\section{Applications}

In this section, we propound two applications in reliability and renewal theory based on  results given in the preceding sections.

\subsection{Reliability}
Let us consider a one-unit system which has the ability to withstand a random number of shocks. We assume that the shocks arrive according to a non-homogeneous Poisson process, and that the number of shocks and the interarrival (or successive) times of shocks are independent. Let $N$ denote the random number of shocks survived by the system, whereas $X_j$ denotes the random interarrival time between the $(j-1)$-th and $j$-th shocks. Hence, the lifetime $T$ of the system is given by $T=\sum_{j=1}^{N}X_j.$
Moreover, let the interarrivals be independent and identically distributed, and let the renewal process describing the number of shocks have cumulative intensity function
$\Lambda(t)=-\log \overline{F}(t)=\int_{0}^{t}\lambda(\tau)\,{\rm d}\tau$, $t\geq 0$, where $\lambda(\tau)$ is the associated
hazard rate (\ref{eq:failrate}). Then, the CDF of $T$ can be written as
\begin{equation}\label{shock}
F_T(t)=\sum_{k=0}^{\infty}P(k)\frac{\Lambda^k(t)}{k!}e^{-\Lambda(t)},\qquad t>0,
\end{equation}
where $P(k)=\mathbb{P}(N\leq k)$, $k\in \mathbb{N}$, is the distribution function of the number of shocks
survived by the device, with $\overline{P}(0)=1-P(0)=1$.      
Relation (\ref{shock}) also holds for a repairable system as discussed in Chahkandi \emph{et al.}\ \cite{Chahkandi}.
\begin{theorem}
Let us consider two devices with random lifetimes $T_1$ and $T_2$ subject to shocks arriving according to a non-homogeneous Poisson process, and let $P_1(k)$ and $P_2(k)$ be respectively the distribution functions of the number of shocks survived by the two devices.
If $N_1\leq_{rhr}N_2,$ then $T_1\leq_{wmit}^\phi T_2.$
\end{theorem}
\begin{proof}
By making use of (\ref{shock}), we have for all $t>0,$
$$
\int_{0}^{t}\phi(x)F_{T_i}(x)\,{\rm d}x
 =\sum_{k=0}^{\infty}P_i(k)\int_{0}^{t}\phi(x)\frac{\Lambda^k(x)}{k!}\overline{F}(x)\,{\rm d}x,\qquad  i=1,2.
$$
From   (ii) of Theorem \ref{equivalent:wmrlthm}, it is sufficient to see that
$\int_{0}^{t}\phi(x)F_{T_2}(x)\,{\rm d}x/\int_{0}^{t}\phi(x)F_{T_1}(x)\,{\rm d}x$
is an increasing function of $t,$ or equivalently $\int_{0}^{t}\phi(x)F_{T_i}(x)\,{\rm d}x$ is $TP_2$ in $(i,t)\in \{1,2\}\times \mathbb{R}^+.$ Since $N_1\leq_{rhr}N_2$ by assumption, then $P_i(k)$ is $TP_2$ in $(i,k)\in  \{1,2\}\times\mathbb{N}.$
On the other hand, it is not hard to see that
\[
\int_{0}^{t}\phi(x)\frac{\Lambda^k(x)}{k!}\overline{F}(x)\,{\rm d}x,
\]
is $TP_2$ in $(t,k)\in \mathbb{R}^+\times\mathbb{N}.$ Then, the general composition theorem of Karlin [27] provides that $\int_{0}^{t}\phi(x)F_{T_i}(x)\,{\rm d}x$ is $TP_2$ in $(i,t)\in \{1,2\}\times \mathbb{R}^+$ and hence the claimed result follows.
\end{proof}
In the special case in which the interarrival times are independent and identically exponentially distributed, one clearly has that $\Lambda^k(t)=(\lambda t)^k$ in the right-hand-side of the distribution function (\ref{shock}). Let us consider the cumulative weight function $\psi(x)=x^r$, i.e.\ the weight function $\phi(x)=rx^{r-1}$, for  $r\in \mathbb{N}.$
\begin{theorem}\label{thm:shockpoisson}
Let $T_1$ and $T_2$ be the random lifetimes of two devices subject to shocks
governed by a homogeneous Poisson process having intensity $\lambda$,
and let $N_i$, $i=1,2$, be the random number of shocks survived by the $i$-th device,
with $P_i(k)= P(N\leq k)$, $k\in \mathbb{N}$. If, for $r\in \mathbb{N}$,
\begin{equation}\label{eq:poissonmonoton}
\frac{\sum_{k=0}^{j-r}{r+k-1\choose k}P_2(k)}
{\sum_{k=0}^{j-r}{r+k-1\choose k}P_1(k)}
\quad\hbox{is increasing in $k\in \mathbb{N}$,}
\end{equation}
then $T_1\leq_{wmit}^\phi T_2,$ for the cumulative weight function $\psi(x)=x^r$.
\end{theorem}
\begin{proof}
It is known that the distribution function of $T_i$, $i=1,2$,  is given by
\begin{equation}\label{shockpoisson}
H_{T_i}(x)=\sum_{k=0}^{\infty}P_i(k)\frac{e^{-\lambda x}(\lambda x)^k}{k!},\qquad x\geq 0.
\end{equation}
Let us consider the following well-known relation
\[
\int_{t}^{\infty}e^{-\lambda x}\frac{\lambda^{k+1} x^{k}}{k!}\,{\rm d}x
=\sum_{j=0}^{k}e^{-\lambda t}\frac{(\lambda t)^j}{j!},\qquad  k\in\mathbb{N}_0,\;\; t>0.
\]
Recalling (\ref{shockpoisson}) and using the aforementioned equation,
after some manipulations we get, for  $r\in \mathbb{N}$  and $i=1,2$,
\begin{eqnarray*}
\hspace{-0.5cm} \int_{0}^{t}rx^{r-1}H_{T_i}(x)\,{\rm d}x \!\!\!\!
&=& \!\!\!\! \int_{0}^{t}rx^{r-1}\sum_{k=0}^{\infty}P_i(k)\frac{e^{-\lambda x}(\lambda x)^k}{k!}\, {\rm d}x
\\
&=& \!\!\!\! \frac{r!}{\lambda^{r}}\sum_{k=0}^{\infty}P_i(k){r+k-1\choose k}\int_{0}^{t}e^{-\lambda x}\frac{\lambda^{k+r}x^{k+r-1}}{(k+r-1)!}\, {\rm d}x
\\
&=& \!\!\!\! \frac{r!}{\lambda^{r}}\sum_{k=0}^{\infty}P_i(k){r+k-1\choose k}\left[1-\int_{t}^{\infty}e^{-\lambda x}\frac{\lambda^{k+r}x^{k+r-1}}{(k+r-1)!}\, {\rm d}x\right]
\\
&=& \!\!\!\! \frac{r!}{\lambda^{r}}\sum_{k=0}^{\infty}P_i(k){r+k-1\choose k}\left[1-\sum_{j=0}^{k+r-1}e^{-\lambda t}\frac{(\lambda t)^j}{j!}\right]
\\
&=& \!\!\!\! \frac{r!}{\lambda^{r}}\sum_{k=0}^{\infty}P_i(k){r+k-1\choose k}\sum_{j=k+r}^{\infty}e^{-\lambda t}\frac{(\lambda t)^j}{j!}
\\
&=& \!\!\!\! \frac{r!}{\lambda^{r}}\sum_{j=r}^{\infty}e^{-\lambda t}\frac{(\lambda t)^j}{j!}\sum_{k=0}^{j-r}{r+k-1\choose k}P_i(k).
\end{eqnarray*}
Since $e^{-\lambda t}{(\lambda t)^j}/{j!}$ is $TP_2$ in $(j,t)\in \mathbb{N}\times \mathbb{R}^+,$ and
recalling relation (\ref{eq:poissonmonoton}), the general composition theorem of Karlin \cite{Karlin}
implies that $\int_{0}^{t}rx^{r-1}H_{T_i}(x)\,{\rm d}x$ is $TP_2$ in
$(i,t)\in \{1,2\}\times \mathbb{R}^+.$ This is equivalent to state that $T_1\leq_{wmit}^\phi T_2$
for $\psi(x)=x^r$.
\end{proof}
\par
We remark that
the case concerning the weight function $\phi(x)=x$  is considered in Theorem 14 of Kayid and Izadkhah
\cite{KayidIzadkhah2014}.
\par
Let us now consider another application.
Let $X_1,X_2,\ldots$ be a sequence of i.i.d.\ random variables, and let $N$ be a positive integer-valued random variable, which is independent of the $X_i.$ Denote by
\[
X_{N:N}=\max\{X_1,X_2,\ldots,X_N\}
\]
the maximum extreme order statistic in a sample having random  size. This random variable arises naturally in reliability theory as the lifetime of a parallel system with the random number of identical components with lifetimes $X_1,X_2,\ldots,X_N.$ In life testing, if a random censoring is adopted, then the completely observed data constitute a sample of random size $N,$ say, where $X_1,X_2,\ldots,X_N,$ $N>0,$ is a random variable of integer values. Let $X_{N_i:N_i}$ denote the maximum order statistic among
$X_1,X_2,\ldots,X_{N_i},$ where $N_i$ is a positive integer-valued random variable which is independent from the sequence of $X_1,X_2,\ldots$ for each $i=1,2.$ Now, we have the following theorem.
\begin{theorem}
Let the weight function $\phi(x)$ be increasing in $x.$ If $N_1\leq_{hr}N_2$, then $X_{N_1:N_1}\leq_{wmit}^\phi X_{N_2:N_2}.$
\end{theorem}
\begin{proof}
Denote by $H_{N_i:N_i}(t)$ the distribution function of $X_{N_i:N_i}$ given as
\[
H_{N_i:N_i}(t)=\sum_{k=1}^{\infty}p^{i}_k\,F^k(t),\quad \hbox{for all $t>0,$}
\]
where $F(t)$ is the common cumulative distribution function of the $X_i$ and $p^i_k=P(N_i=k)$, $k\in \mathbb{N},$ is the probability mass function of $N_i,\ i=1,2.$ Clearly, $F^k(t)$ is the CDF of $X_{N:N}$ conditional on $N=k$. It is not hard to see that for all $t>0$ and for each $i=1,2$ one has
\[
\varphi(t,i)=\int_{0}^{t}\phi(x){H}_{N_i:N_i}(x){\rm d}x=\sum_{k=1}^{\infty}\eta(t,k)\rho(k,i),
\]
where $\eta(t,k)=\int_{0}^{t}\phi(x)F^k(x){\rm d}x,$ and $\rho(k,i)=p^i_k.$ Denote $\nu(k,i)=\sum_{j=k}^{\infty}p^i_j,$
for each $k\in \mathbb{N}$ and $i=1,2.$ Assumption $N_1\leq_{hr}N_2$ (inequality $\leq_{hr}$
stands for the hazard rate order between $N_1$ and $N_2$)
implies that $\nu(k,i)$ is $TP_2$ in $(k,i)\in \mathbb{N}\times \{1,2\}.$
On the other hand, $\eta(t,k)$ is $TP_2$ in $(t,k)\in \mathbb{R}^+\times \mathbb{N}.$
Applying Lemma 2.1 in Ortega [11] gives $\varphi(t,i)$ is $TP_2$ in $(t,i)\in \mathbb{R}^+\times \{1,2\},$
which is equivalent to say that $X_{N_1:N_1}\leq_{wmit}^\phi X_{N_2:N_2}$.
\end{proof}
\subsection{Renewal Theory}
Let us consider a renewal process with i.i.d.\ non-negative interarrival times $\{X_n\}_{n\in \mathbb{N}}$ having common
distribution function $F(t)$ and finite mean $\mu=\mathbb E[X_n]$. Let $S_n=\sum_{i=1}^{n}X_i$, $n\in \mathbb N$,
with $S_0\equiv 0$, be the time of the $k$th arrival. We define $N(t)=\max\{n: S_n\leq t\}$,
which represents the number of renewals during $(0,t]$.  The excess lifetime $\gamma(t)=S_{N(t)+1}-t$
at time $t\geq 0$ is the time elapsed from the time $t$ to the first arrival after $t.$
Recall that $\gamma(0)$ has distribution function $F,$ that is, $\gamma(0)\stackrel{d}{=}X_1.$ The expected number of renewals in $(0, t]$ can be obtained as
\begin{equation}\label{Mt}
M(t)=\mathbb{E}[N(t)]=F(t)+\int_{0}^{t}F(t-u)\,{\rm d}M(u).
\end{equation}
It is well-known that the CDF of $\gamma(t)$ is given as
\begin{equation}\label{gammat}
\mathbb{P}[\gamma(t)\leq x]=F(t+x)+\int_{0}^{t}F(t-u+x)\,{\rm d}M(u)-M(t),
\end{equation}
for all $x,t\geq0.$ In the literature, several results have been given to characterize the stochastic orders by the excess lifetime in a renewal process. For more details on definitions and properties, readers are referred to Barlow and Proschan \cite{Barlow}. Next, we will investigate the behavior of the excess lifetime of a renewal process with WMIT interarrivals. We recall that the CDF of the residual lifetime (\ref{eq:residuallifetime}) is given by
\[
F_t(x)=P(X-t\leq x|X>t)=\frac{F(t+x)-F(t)}{1-F(t)},\quad t>0.
\]
Moreover, we say $X$ is new better than used (NBU) if $X_t\leq_{st}X$ for all $t>0,$ where $X_t$ is the residual lifetime defined in (\ref{eq:residuallifetime}).
\begin{theorem}
Let $X_t\leq_{wmit}^\phi X$ for all $t>0.$ If $X$ is IWMIT and is NBU, then $\gamma(t)\leq_{wmit}^\phi \gamma(0)$ for all $t>0.$
\end{theorem}
\begin{proof}
Since $X_t\leq_{wmit}^\phi X$ for all $t>0,$ it follows that
$$
\int_{0}^{s}\phi(x)[F(t+x)-F(t)]\,{\rm d}x\geq [F(t+s)-F(t)]\int_{0}^{s}\phi(x)\frac{F(x)}{F(s)}\,{\rm d}x,
$$
for all $s>0.$ By (\ref{Mt}) and (\ref{gammat}), we have that
\begin{eqnarray*}
  &&   \hspace{-0.5cm}\int_{0}^{s}\phi(x)\mathbb{P}[\gamma(t)\leq x]\,{\rm d}x \\
  &=&\int_{0}^{s}\phi(x)[F(t+x)-F(t)]\,{\rm d}x+\int_{0}^{s}\int_{0}^{t}\phi(x)[F(t-u+x)-F(t-u)]{\rm d}M(u)\,{\rm d}x \\
  &=& \int_{0}^{s}\phi(x)[F(t+x)-F(t)]\,{\rm d}x+\int_{0}^{t}\int_{0}^{s}\phi(x)[F(t-u+x)-F(t-u)]\,{\rm d}x\,{\rm d}M(u) \\
  &\geq&\int_{0}^{s}\phi(x)[F(t+x)-F(t)]\,{\rm d}x+\int_{0}^{t}[F(t-u+s)-F(t-u)]\int_{0}^{s}\phi(x)\frac{F(x)}{F(s)}\,{\rm d}x\,{\rm d}M(u)\\
  &=&\int_{0}^{s}\phi(x)[F(t+x)-F(t)]\,{\rm d}x+
  \int_{0}^{s}\phi(x)\frac{F(x)}{F(s)}\,{\rm d}x\int_{0}^{t}\left[F(t-u+s)-F(t-u)\right]\,{\rm d}M(u)\\
  &=&\int_{0}^{s}\phi(x)[F(t+x)-F(t)]\,{\rm d}x+
  \int_{0}^{s}\phi(x)\frac{F(x)}{F(s)}\,{\rm d}x\left[P(\gamma(t)\leq s)-F(t+s)+F(t)\right]\\
  &\geq&[F(t+x)-F(t)]\int_{0}^{s}\phi(x)\frac{F(x)}{F(s)}\,{\rm d}x+
  \int_{0}^{s}\phi(x)\frac{F(x)}{F(s)}\,{\rm d}x\left[P(\gamma(t)\leq s)-F(t+s)+F(t)\right]\\
  &=&\int_{0}^{s}\phi(x)\frac{F(x)}{F(s)}{\rm d}x\,\mathbb{P}[\gamma(t)\leq s].
\end{eqnarray*}
Hence, it holds that for all $t,s\geq0$,
\[
\int_{0}^{s}\phi(x)\frac{\mathbb{P}[\gamma(t)\leq x]}{\mathbb{P}[\gamma(t)\leq s]}\,{\rm d}x\geq \int_{0}^{s}\phi(x)\frac{F(x)}{F(s)}\,{\rm d}x,
\]
which means that $\gamma(t)\leq_{wmit}^\phi \gamma(0)$ for all $t>0.$
\end{proof}

\section{Concluding remarks}

It is of interest for the industry to perform systematic studies using reliability concepts in view of economic repercussions and safety issues.
Due to the existence of a great number of scenarios, a statistical comparison of reliability measures is desired in several applied contexts,
such as reliability engineering and biomedical fields.
For this reason, we have introduced a  stochastic order based on the MIT function, named  weighted mean inactivity time (WMIT) order,
which is dual to the weighted mean residual life order. The relationship of this new order with other well-known stochastic orders has been discussed.
It was shown that the WMIT order lies in the framework of the RHR and the MIT orders under suitable conditions, and hence it enjoys several
useful properties which can be applied in reliability and survival analysis. Moreover, we also discussed its monotonicity properties.
Further, we used the WMIT to determine the expressions for the variance of transformed random variable as well as the weighted GCE.
Among the several results on such measures, we provided some characterizations and preservation properties of the new order
under shock models, random maxima and renewal theory. Our results provide new concepts and applications in reliability, statistics, and risk  theory.
\par
Further properties and applications of the new stochastic order and the new proposed class will be the
object of future investigations. For example, the result of this paper can be extended to the doubly truncated (interval) random variables. Specifically, given the random lifetime $X$ and the cumulative weighted random variable $\psi(X),$ one can consider 
$$
 [\psi(X)- \psi(t_1)|t_1\leq X \leq t_2]
 \qquad  \hbox{and} \qquad 
 [\psi(t_2) - \psi(X)|t_1\leq X \leq t_2]
$$ 
where $(t_1, t_2) \in D^{\star}=\{(t_1, t_2) : F(t_1)<F(t_2)\}.$ Given that the lifetime having age $t_1$ will expire before age $t_2$, 
the first random variable is related to the remaining lifetime, whereas the second one is related to the inactivity time  
(see e.g.\ Sankaran and Sunoj \cite{Sankaran}, Khorashadizadeh \emph{et al.}\ \cite{Khorashadizadeh} and  
references therein).

%

\begin{thebibliography}{99}
%
\bibitem{Ahmad}
Ahmad IA, Kayid M, Pellerey F, 
Further results involving the MIT order and the IMIT class, 
Probab Engrg Inform Sci, 19, 377--395 (2005) 
%
\bibitem{Badia}
Badia F, Berrade M,  On the reversed hazard rate and mean inactivity time of mixture. 
In Advances in Mathematical Modeling for Reliability, T. Bedford, Ed. et al. Amsterdam, 
103--110, The Netherlands: Delft Univ. Press (2008)
%
\bibitem{Barlow-Proschan}
Barlow RE, Proschan F. 
Mathematical Theory of Reliability. New York, NY, USA, Wiley (1965) 
%
\bibitem{Barlow}
Barlow RE,  Proschan F.  
Statistical Theory of Reliability and Life testing. New York, Holt, Rinehart and Winston (1975)
%
\bibitem{CalMaria}
Cal\`{\i} C,  Longobardi M,  
Some mathematical properties of the ROC curve and their applications,
Ricerche Mat. 64, 391--402 (2015)
%
\bibitem{Chahkandi}
Chahkandi M, Ahmadi J, Baratpour S, 
Some results for repairable systems with minimal repairs,
Appl Stoch Models Bus Ind  30, 218--226 (2014)
%
\bibitem{Denuit}
Denuit M, Dhaene J, Goovaerts M, Kaas R, 
Actuarial Theory for Dependent Risks Measures, Orders and Models. 
John Wiley \& Sons (2006) 
%
\bibitem{DiCrescenzo2002}
Di Crescenzo A, Longobardi M, 
Entropy-based measure of uncertainty in past lifetime distributions,
J Appl Probab  39, 434--440 (2002) 
%
\bibitem{DiLongobardi2006}
Di Crescenzo A, Longobardi M,  
On weighted residual and past entropies,
Sci Math Jpn  64,  255--266 (2006)
%
\bibitem{Di-Longobardi-2009}
Di Crescenzo A, Longobardi M, 
On cumulative entropies,
J Statist Plann Inference, 139, 4072--4087 (2009)
%
\bibitem{DMM}
Di Crescenzo A, Martinucci B, Mulero J, 
A quantile-based probabilistic mean value theorem,
Probab Engrg Inform Sci, 30, 261--280 (2016)
%
\bibitem{Di-Toomaj}
Di Crescenzo A, Toomaj A, 
Further results on the generalized cumulative entropy,
Kybernetika, 53, 959--982 (2017) 
%
\bibitem{DiPaol}
Di Crescenzo A, Paolillo L, 
Analysis and applications of the residual varentropy of random lifetimes, 
Probab Engrg Inform Sci (2020) 
DOI: https://doi.org/10.1017/S0269964820000133
%
\bibitem{Ebrahimi}
Ebrahimi N, 
How to measure uncertainty in the residual life time distribution,
Sankhya A, 58, 48--56 (1996) 
%
\bibitem{Fernandez-Ponce1998}
Fernandez-Ponce JM, Kochar SC, Mu$\tilde{\hbox{n}}$oz-Perez J,  
Partial orderings of distributions based on right spread functions,
J Appl Probab, 35, 221--228 (1998)
%
\bibitem{Finkelstein}
Finkelstein MS, 
On the reversed hazard rate,
Reliab Eng Syst Safe, 78, 71--75. (2002)
%
\bibitem{Fradelizi}
Fradelizi M, Madiman M,  Wang L, 
Optimal concentration of information content for logconcave densities. 
In C. Houdre, D. Mason, P. Reynaud-Bouret and J. Rosinski (eds.), 
High Dimensional Probability VII. Progress in Probability, vol. 71, 45--60, Cham, Springer (2016) 
%
\bibitem{Goli}
Goliforushani S, Asadi M, 
On the discrete mean past lifetime, 
Metrika, 68, 209--217 (2008) 
%
\bibitem{GuptaPeng}
Gupta RC, Peng C, 
Estimating reliability in proportional odds ratio models,
Comp Stat Data Anal, 53,1495--1510 (2009)
%
\bibitem{IzadkhahKayid}
Izadkhah S, Kayid M, 
Reliability analysis of the harmonic mean inactivity time order, 
IEEE Tr Reliab, 62, 329--337 (2013) 
%
\bibitem{Jewitt}
Jewitt I, 
Choosing between risky prospects: The characterization of comparative statics results, 
and location independent risk, Management Sci, 35, 60--70 (1989) 
%
\bibitem{Karlin}
Karlin S, 
Total Positivity. Stanford, CA, USA: Stanford Univ. Press (1968)
%
\bibitem{Kayal}
Kayal S, 
On generalized cumulative entropies, 
Probab Engrg Inform Sci, 30, 640--662 (2016) 
%
\bibitem{KayalMoharana}
Kayal S, Moharana SR, 
A shift-dependent generalized cumulative entropy of order $n$.
Comm. Statist. Simulation Comput, 48, 1768--1783 (2018) 
%
\bibitem{KayidAhmad2004}
Kayid M, Ahmad IA, 
On the mean inactivity time ordering with reliability applications,
Probab Engrg Inform Sci, 18, 395--409 (2004)
%
\bibitem{KayidIzadkhah2014}
Kayid M, Izadkhah S, 
Mean inactivity time function, associated orderings, and classes of life distributions, 
IEEE Trans Reliab, 63, 593--602 (2014) 
%
\bibitem{Khorashadizadeh}
Khorashadizadeh M, Rezaei Roknabadi AH,  Mohtashami Borzadaran GR, 
Doubly truncated (interval) cumulative residual andpast entropy, 
Stat Probab Lett, 83, 1464--1471 (2013) 
%
\bibitem{Kundu}
Kundu C, Nanda AK, 
Some reliability properties of the inactivity time, 
Comm Statist Theory Methods, 39, 899--911 (2010) 
%
\bibitem{Mirali2017b}
Mirali M, Baratpour S, 
Some results on weighted cumulative entropy, 
J Iran Stat Soc, 17, 21--32 (2017) 
%
\bibitem{Misagh2011}
Misagh F, Panahi Y, Yari GH,  Shahi R, 
Weighted Cumulative entropy and its estimation. 
In: 2011 IEEE International Conference on Quality and Reliability, ICQR (2011)  
doi:10.1109/ICQR.2011.6031765
%
\bibitem{Muliere1993}
Muliere P, Parmigiani G, Polson NG, 
A note on the residual entropy function,
Probab Engrg Inform Sci, 7, 413--420 (1993)
%
\bibitem{Navarro-2010}
Navarro J, del Aguila Y, Asadi M, 
Some new results on the cumulative residual entropy, 
J Statist Plann Infer, 140, 310--322 (2010)
%
\bibitem{Sankaran}
Sankaran PG, Sunoj SM, 
Identification of models using failure rate and mean
residual life of doubly truncated random variables, 
Statistical Papers, 45, 97--109 (2004) 
%
\bibitem{Shaked}
Shaked M, Shanthikumar JG, 
Stochastic Orders and Their Applications, 
Academic Press, San Diego (2007)
%
\bibitem{Yasaei}
Suhov Y, Yasaei S, 
Weighted cumulative entropies: an extension of CRE and CE (2015)  
arXiv: 1507.07051v1[cs.IT]
%
\bibitem{Tahmasebietal2019}
Tahmasebi S, Longobardi M, Foroghi F, Lak F, 
An extension of weighted generalized cumulative past measure of information,
Ricerche Matem, 69, 53--81 (2020)
%
\bibitem{ToomajAD2020}
Toomaj A, Di Crescenzo A, 
Generalized entropies, variance and applications.
Entropy, 22, 709 (2020) 
%
\bibitem{ToomajAD2020Mathematics}
Toomaj A, Di Crescenzo A, 
Connections between weighted generalized cumulative residual entropy and variance,
Mathematics, 8, 1072 (2020) 
%
\end{thebibliography}
\end{document}